\documentclass[lineno]{biometrika}

\usepackage{amsmath, amssymb}
\usepackage{caption}

\usepackage{times}
\usepackage{bm}
\usepackage{natbib}
\usepackage{mathtools}
\usepackage{subfigure}

\usepackage[plain,noend]{algorithm2e}

\makeatletter
\renewcommand{\algocf@captiontext}[2]{#1\algocf@typo. \AlCapFnt{}#2} 
\def\@algocf@capt@plain{top}
\renewcommand{\algocf@makecaption}[2]{%
  \addtolength{\hsize}{\algomargin}%
  \sbox\@tempboxa{\algocf@captiontext{#1}{#2}}%
  \ifdim\wd\@tempboxa >\hsize
    \hskip .5\algomargin%
    \parbox[t]{\hsize}{\algocf@captiontext{#1}{#2}}
  \else%
    \global\@minipagefalse%
    \hbox to\hsize{\box\@tempboxa}
  \fi%
  \addtolength{\hsize}{-\algomargin}%
}
\makeatother

\def\T{{ \mathrm{\scriptscriptstyle T} }}

\def\var{\textnormal{var}}
\def\cov{\textnormal{cov}}
\def\sumN{\sum_{i=1}^N}
\def\sumJ{\sum_{j=1}^J}

\def\asim{\overset{.}{\sim}}

\def\obs{\textnormal{obs}}

\def\digamma{\text{digamma}}

\def\-{\mbox{-}}
\def\SSTre{\textup{SSTre}}
\def\SSRes{\textup{SSRes}}
\def\MSTre{\textup{MSTre}}
\def\MSRes{\textup{MSRes}}
\def\FRT{\text{Fisher randomization test~}}
\def\pr{\textup{pr}}
\def\HN{H_0(\textup{Neyman})}
\def\HF{H_0(\textup{Fisher})}

\begin{document}

\jname{Biometrika}
\jyear{}
\jvol{}
\jnum{}
\copyrightinfo{\Copyright\ 2012 Biometrika Trust\goodbreak {\em Printed in Great Britain}}


\markboth{Peng Ding \and Tirthankar Dasgupta}{A randomization-based perspective of analysis of variance}

\title{A randomization-based perspective of analysis of variance: a test statistic robust to treatment effect heterogeneity}

\author{Peng Ding}
\affil{Department of Statistics, University of California, Berkeley, 425 Evans Hall, Berkeley, California 94720, U.S.A. \email{pengdingpku@berkeley.edu}}

\author{\and Tirthankar Dasgupta}
\affil{Department of Statistics and Biostatistics, Rutgers University, 110 Frelinghuysen Road, Piscataway, New Jersey 08901, U.S.A. \email{tirthankar.dasgupta@rutgers.edu}}

\maketitle

\begin{abstract}
Fisher randomization tests for Neyman's null hypothesis of no average treatment effects are considered in a finite population setting associated with completely randomized experiments with more than two treatments. The consequences of using the $F$ statistic to conduct such a test are examined both theoretically and computationally, and it is argued that under treatment effect heterogeneity, use of the $F$ statistic in the Fisher randomization test can severely inflate the type I error under Neyman's null hypothesis. An alternative test statistic is proposed, its asymptotic distributions under Fisher's and Neyman's null hypotheses are derived, and its advantages demonstrated.
\end{abstract}

\begin{keywords}
Additivity; Fisher randomization test; Null hypothesis; One-way layout
\end{keywords}

\vspace{-.5 in}

\section{Introduction}
\label{sec:introduction}
One-way analysis of variance \citep{fisher1925statistical} is arguably the most commonly used tool to analyze completely randomized experiments with more than two treatments. The standard $F$ test for testing equality of mean treatment effects can be justified either by assuming a linear additive super population model with identically and independently distributed normal error terms, or by using the asymptotic randomization distribution of the $F$ statistic. As observed by many experts, units in most real-life experiments are rarely random samples from a super population, making a finite population randomization-based perspective on inference important \citep[e.g.][]{rosenbaum2010design, imbens::2015book, dasgupta::2015}. Fisher randomization tests are useful tools for such inference, because they pertain to a finite population of units, and assess the statistical significance of treatment effects without making any assumptions about the underlying distribution of the outcome. 

In causal inference from finite population, two types of hypotheses are of interest: Fisher's sharp null hypothesis of no treatment effect on any experimental unit \citep{fisher::1935, rubin::1980}, and Neyman's null hypothesis of no average treatment effect \citep{neyman::1923, neyman::1935}. These hypotheses are equivalent without treatment effect heterogeneity \citep{ding2016randomization} or equivalently under the assumption of strict additivity of treatment effects, i.e., the same treatment effect for each unit \citep{kempthorne::1952}. In the context of a multi-treatment completely randomized experiment, Neyman's null hypothesis allows for treatment effect heterogeneity, which is weaker than Fisher's null hypothesis and is of greater interest. We find that the Fisher randomization test using the $F$ statistic can inflate the type I error under Neyman's null hypothesis, when the sample sizes and variances of the outcomes under different treatment levels are negatively associated. We propose to use the $X^2$ statistic defined in \S \ref{sec::newstatistic}, a statistic robust to treatment effect heterogeneity, because the resulting Fisher randomization test is exact under Fisher's null hypothesis and controls asymptotic type I error under Neyman's null hypothesis.


\section{Completely randomized experiment with $J$ treatments}
\label{sec:basic}

Consider a finite population of $N$ experimental units, each of which can be exposed to any one of $J$ treatments. 
Let $Y_i(j)$ denote the potential outcome \citep{neyman::1923} of unit $i$ when assigned to treatment level $j$ ($i=1,\ldots, N;j=1,\ldots, J).$ For two different treatment levels $j$ and $j'$, we define the unit-level treatment effect as
$\tau_i(j,j') = Y_i(j) - Y_i(j')$, 
and the population-level treatment effect as
\begin{equation*}
\tau(j,j') = N^{-1}\sumN \tau_i(j,j')  = N^{-1} \sumN \{  Y_i(j) - Y_i(j')\}  \equiv \bar{Y}_{\cdot}(j) - \bar{Y}_{\cdot}(j'), 
\end{equation*}
where $\bar{Y}_{\cdot}(j) = N^{-1} \sumN Y_i(j)$ is the average of the $N$ potential outcomes for treatment $j$.

The treatment assignment mechanism can be represented by the binary random variable $W_i(j),$ which equals $1$ if the $i$th unit is assigned to treatment $j$, and $0$ otherwise. Equivalently, it can be represented by the discrete random variable $W_i = \sumJ j W_i(j) \in \{ 1,\ldots,J\}$, the treatment received by unit $i$.
Let $  (W_1, \ldots, W_N)$ be the treatment assignment vector, and let $  (w_1,\ldots, w_N)$ denote its realization. For the $N=\sumJ N_j$ units, $(N_1,\ldots, N_J)$ are assigned at random to treatments $(1,\ldots, J)$ respectively, the treatment assignment mechanism satisfies
$\pr\{  (W_1, \ldots, W_N) =   (w_1,\ldots, w_N) \} = \prod_{j=1}^J N_j!/N!$
if $\sumN W_i(j) = N_j$, and $0$ otherwise. The observed outcomes are deterministic functions of the treatment received and the potential outcomes, given by
$ Y_i^\obs = \sumJ W_i(j) Y_i(j) \ (i=1, \ldots, N).$

\section{The Fisher randomization test under the sharp null hypothesis}
\label{sec:FRT}

\citet{fisher::1935} was interested in testing the following sharp null hypothesis of zero individual treatment effects:
\begin{equation*}
\HF: Y_i(1)   = \cdots = Y_i(J), \quad (i=1,\ldots, N).  
\end{equation*}
Under $\HF$, all the $J$ potential outcomes $Y_i(1), \ldots, Y_i(J)$ are equal to the observed outcome $Y_i^\obs$, for all units $i = 1, \ldots, N$. Thus any possible realization of the treatment assignment vector would generate the same vector of observed outcomes. This means, under $\HF$ and given any realization $(W_1,\ldots, W_N) = (w_1, \ldots, w_N)$, the observed outcomes are fixed. Consequently, the randomization distribution or null distribution of any test statistic, which is a function of the observed outcomes and treatment assignment vector, is its distribution over all possible realizations of the treatment assignment. The $p$-value is the tail probability measuring the extremeness of the test statistic with respect to its randomization distribution.
Computationally, we can enumerate or simulate a subset of all possible randomizations to obtain this randomization distribution of any test statistic and thus perform the Fisher randomization test \citep{fisher::1935, imbens::2015book}. \citet{fisher1925statistical} suggested using the $F$ statistic to test the departure from $\HF$.
Define $\bar{Y}_{\cdot}^\obs(j) = N_j^{-1} \sum_{i=1}^N W_i(j)Y_i^\obs $ as the sample average of the observed outcomes within treatment level $j$, and $\bar{Y}_{\cdot}^\obs = N^{-1} \sum_{i=1}^N Y_i^\obs $ as the sample average of all the observed outcomes.
Define $ s^2_{\obs}(j) =   ( N_j - 1 )^{-1}   \sumN W_i(j)  \{ Y_i^\obs - \bar{Y}_{\cdot}^\obs (j) \}^2 $ and $s^2_\obs = (N-1)^{-1} \sumN (Y_i^\obs - \bar{Y}_{\cdot}^\obs)^2 $ as the corresponding sample variances with divisors $N_j-1$ and $N-1$, respectively.
Let 
\begin{equation*}
\SSTre = \sumJ N_j \{ \bar{Y}_{\cdot}^\obs(j)  - \bar{Y}_{\cdot}^\obs \}^2  
\end{equation*}
be the treatment sum of squares, and let
\begin{equation*}
\SSRes = \sumJ \sum_{i:W_i(j) = 1}  \{ Y_i^\obs - \bar{Y}_{\cdot}^\obs (j) \}^2
=  \sumJ  (N_j - 1) s^2_{\obs}(j)  
\end{equation*}
be the residual sum of squares. The treatment and residual sums of squares sum up to the total sum of squares $\sumN (Y_i^\obs - \bar{Y}_{\cdot}^\obs)^2 = (N-1)s^2_\obs $.
The $F$ statistic
\begin{equation}
F = \frac{   \SSTre / (J-1)  }{   \SSRes / (N-J)    }  \equiv  \frac{ \MSTre }{ \MSRes} \label{eq:F}
\end{equation}
 is defined as the ratio of the mean squares of treatment $ \MSTre = \SSTre / (J-1) $ to the mean squares of residual $\MSRes=  \SSRes / (N-J)  $.

The distribution of (\ref{eq:F}) under $\HF$ can be well approximated by an $F_{J-1, N-J}$ distribution with degrees of freedom $J-1$ and $N-J$, as is often used in the analysis of variance table obtained from fitting a normal linear model. Whereas it is relatively easy to show that (\ref{eq:F}) follows $F_{J-1, N-J}$ if the observed outcomes follows a normal linear model drawn from a super population, arriving at such a result using a purely randomization-based argument is non-trivial. Below, we state a known result on the approximate randomization distribution of (\ref{eq:F}), in which we use the notation $A_N \asim B_N$ to represent two sequences of random variables $\{A_N\}_{N=1}^\infty$ and $\{B_N\}_{N=1}^\infty$ that have the same asymptotic distribution as $N \rightarrow \infty$. Throughout our discussion, we assume the following regularity conditions required by the finite population central limit theorem for causal inference \citep{li2017general}.

\begin{condition}
As $N\rightarrow \infty$, for all $j$, $N_j/N$ has a positive limit, $\bar{Y}_{\cdot}(j)$ and $S_{\cdot}^2(j)$ have finite limits, and $N^{-1}\max_{1\leq i\leq N} | Y_i(j) - \bar{Y}_{\cdot}(j) |^2 \rightarrow 0.$
\end{condition}

\begin{theorem}
\label{thm::fisher-sharp-null}
Assume $\HF$. Over repeated sampling of $(W_1,\ldots, W_N)$, the expectations of the residual and treatment sums of squares are $E(\SSTre) = (J-1) s^2_\obs$ and $E(\SSRes) = (N-J) s^2_{\obs}$, and as $N\rightarrow \infty$, the asymptotic distribution of (\ref{eq:F}) is
$$
F\asim \frac{  \chi^2_{J-1}/ (J-1) }{   \{  (N-1) - \chi^2_{J-1}  \}/(N-J) } \asim F_{J-1, N-J}.
$$
\end{theorem}

\begin{remark} \label{remark::theorem1_1}
As $N\rightarrow \infty$, both the statistic $F$ and random variable $F_{J-1, N-J}$ are asymptotically $\chi^2_{J-1}/(J-1)$. The original $F$ approximation for randomization inference for a finite population was derived by cumbersome moment matching between statistic (\ref{eq:F}) and the corresponding $F_{J-1, N-J}$ distribution \citep{welch::1937, pitman::1938, kempthorne::1952}. Similar to \citet{silvey1954asymptotic}, we provide a simpler proof based on the finite population central limit theorem in the Supplementary Material. 
\end{remark}

\begin{remark} \label{remark::theorem1_2}
Under $\HF$, the total sum of squares is fixed, but its components $\SSTre$ and $\SSRes$ are random through the treatment assignment $(W_1,\ldots,W_N)$, and their expectations are calculated with respect to the distribution of the treatment assignment. Also, the ratio of expectations of the numerator $\MSTre$ and denominator $\MSRes$ of (\ref{eq:F}) is $1$ under $\HF$.
\end{remark}

\section{Sampling properties of the F statistic under Neyman's null hypothesis}
\label{sec::Neyman}

In Section \ref{sec:FRT}, we discussed the randomization distribution, i.e., the sampling distribution under $\HF$, of the $F$ statistic in (\ref{eq:F}). However, the sampling distribution of the $F$ statistic under Neyman's null hypothesis of zero average treatment effect \citep{neyman::1923, neyman::1935}, i.e.,
\begin{equation*}
\HN: \bar{Y}_{\cdot}(1) = \cdots = \bar{Y}_{\cdot}(J),  
\end{equation*}
is often of major interest but is under-investigated \citep{imbens::2015book}. $\HN$ imposes weaker restrictions on the potential outcomes than $\HF$, making it impossible to compute the exact, or even approximate distribution of the $F$ statistic under $\HN$. However, analytical expressions for $E(\SSTre)$ and $E(\SSRes)$ can be derived under $\HN$ along the lines of Theorem \ref{thm::fisher-sharp-null}, and can be used to gain insights about the consequences of testing $\HN$ using the \FRT with the $F$ statistic in (\ref{eq:F}).

For treatment level $j = 1, \ldots, J$, define
$
p_j = N_j/N
$
as the proportion of the units, and $S_{\cdot}^2(j) = (N -1)^{-1} \sumN \{   Y_i(j) - \bar{Y}_{\cdot}(j) \}^2$ as the finite population variances of potential outcomes. Let $\bar{Y}_{\cdot}(\cdot) = \sumJ p_j  \bar{Y}_{\cdot}(j) $ and
$
S^2 = \sumJ p_j S_{\cdot}^2(j)
$
be the weighted averages of the finite population means and variances.
The sampling distribution of the $F$ statistic in (\ref{eq:F}) depends crucially on the finite population variance of the unit-level treatment effects
$$
S_{\cdot}^2(j\-j')  = (N-1)^{-1} \sumN  \{   \tau_i(j, j') - \tau(j, j')  \}^2 .
$$

\begin{definition} \label{def:strict_additivity}
The potential outcomes $\{ Y_i(j):  i=1, \ldots, N, \  j = 1, \ldots, J \} $ have strictly additive treatment effects if for all $j \ne j' $, the unit-level treatment effects $\tau_i(j, j')$ are the same for $i=1, \ldots, N$, or equivalently, $ S_{\cdot}^2(j\-j')  = 0$ for all $j \ne j'$.
\end{definition}

\cite{kempthorne1955randomization} obtained the following result on the sampling expectations of $\SSRes$ and $\SSTre$ for balanced designs with $p_j=1/J$ under the assumption of strict additivity:
\begin{eqnarray}
E(\SSRes) =  (N-J )S^2,\quad
E(\SSTre)  =  \frac{N}{J } \sumJ  \{  \bar{Y}_{\cdot}(j) - \bar{Y}_{\cdot}(\cdot)  \}^2 +   (J-1) S^2.
\label{eq::kempthorne}
\end{eqnarray}
This result implies that with balanced treatment assignments and strict additivity, $E(\MSRes - \MSTre) = 0$ under $\HN$, and provides a heuristic justification for testing $\HN$ using the \FRT with the $F$ statistic. However, strict additivity combined with $\HN$ implies $\HF$, for which this result is already known by Theorem \ref{thm::fisher-sharp-null}. We will now derive results that do not require the assumption of strict additivity, and thus are more general than those in \cite{kempthorne1955randomization}. For this purpose, we introduce a measure of deviation from additivity. Let
\begin{equation*}
 \Delta =  \mathop{\sum\sum}_{j  <  j'} p_j p_{j'}   S_{\cdot}^2(j\-j') 
\end{equation*}
be a weighted average of the variances of unit-level treatment effects. By Definition \ref{def:strict_additivity}, $\Delta = 0$ under strict additivity. If strict additivity does not hold, i.e., there is treatment effect heterogeneity \citep{ding2016randomization}, then $ \Delta \neq 0$. Thus $\Delta$ is a measure of deviation from additivity and plays a crucial role in the following results on the sampling distribution of the $F$ statistic.

\begin{theorem}
\label{thm::sampling-F-Neyman}
Over repeated sampling of $(W_1, \ldots, W_N)$, the expectation of the residual sum of squares is
$
E(\SSRes)
= \sumJ (N_j - 1)  S_{\cdot}^2(j) ,
$
and the expectation of the treatment sum of squares is 
$$
E(\SSTre) =
\sumJ N_j  \left\{ \bar{Y}_{\cdot}(j)  - \bar{Y}_{\cdot}(\cdot)  \right\}^2+ \sumJ (1-p_j) S_{\cdot}^2(j) - \Delta,
$$
which reduces to $E(\SSTre) =   \sumJ (1-p_j) S_{\cdot}^2(j) - \Delta$ under $\HN$. 
\end{theorem}

\begin{corollary}
\label{coro::1}
Under $\HN$ with strict additivity in Definition \ref{def:strict_additivity}, or, equivalently, under $\HF$, the above results reduce to
$
E(\SSRes) = (N-J) S^2
$
and
$
E(\SSTre) =  (J-1) S^2,
$
which coincide with Theorem \ref{thm::fisher-sharp-null}.
\end{corollary}

\begin{corollary}
\label{coro::2}
For a balanced design with $p_j = 1/J$,
$$
E(\SSRes) =  (N-J )S^2,\quad
E(\SSTre)  =  \frac{N}{J  } \sumJ  \{  \bar{Y}_{\cdot}(j) - \bar{Y}_{\cdot}(\cdot)  \}^2 +   (J-1) S^2 - \Delta.
$$
Furthermore, under $\HN$, 
$
E(\SSRes) =  (N-J )S^2
$
and
$
E(\SSTre)  = (J-1) S^2 - \Delta,
$
implying that the difference between the mean squares of the residual and the treatment is
$
E ( \MSRes - \MSTre )  =  \Delta  / (J - 1)  \geq 0.
$
\end{corollary}

The result in \eqref{eq::kempthorne} is a special case of Corollary \ref{coro::2} for $\Delta=0$. Corollary \ref{coro::2} implies that, for balanced designs, if the assumption of strict additivity does not hold, then testing $\HN$ using the \FRT with the $F$ statistic may be conservative, in a sense that it may reject a null hypothesis less often than the nominal level. However, for unbalanced designs, the conclusion is not definite, as will be seen from the following result.

\begin{corollary}
\label{coro::unbalanced}
Under $\HN$, the difference between the mean squares of the residual and the treatment is
$$
E( \MSRes - \MSTre)  =
\frac{  (N-1)J }{  (J-1)(N-J) } \sumJ (p_j - J^{-1})S_{\cdot}^2(j) +  \frac{ \Delta} { J-1}.
$$
\end{corollary}

Corollary \ref{coro::unbalanced} shows that the mean square of the residual may be bigger or smaller than that of the treatment, depending on the balance or lack thereof of the experiment and the variances of the potential outcomes. Under $\HN$, when the $p_j$'s and $S_{\cdot}^2(j)$'s are positively associated, the \FRT using $F$ tends to be conservative; when the $p_j$'s and $S_{\cdot}^2(j)$'s are negatively associated, the \FRT using $F$ may not control correct type I error.

\section{A test statistic that controls type I error more precisely than $F$}
\label{sec::newstatistic}

To address the failure of the $F$ statistic to control type I error of the Fisher randomization test under $\HN$ in unbalanced experiments, we propose to use the following $X^2$ test statistic for the Fisher randomization test. Define $\hat{Q}_j = N_j /   s^2_{\obs}(j) $,
and define the weighted average of the sample means as
$
\bar{Y}^\obs_w = \sumJ  \hat{Q}_j  \bar{Y}_{\cdot}^\obs(j) /    \sumJ \hat{Q}_j.
$
Define the $X^2$ test statistic as
\begin{equation}
X^2 = \sumJ   \hat{Q}_j  \left\{      \bar{Y}_{\cdot}^\obs(j) -    \bar{Y}^\obs_w   \right\}^2 , \label{eq:X^2}
\end{equation}
which can be obtained from weighted least squares. This test statistic has been exploited in the classical analysis of variance literature \citep[e.g.,][]{ james1951comparison, welch1951comparison, johansen1980welch, rice1989one, weerahandi1995anova, krishnamoorthy2007parametric} based on the normal linear model with heteroskedasticity, and a similar idea called studentization has been adopted in the permutation test literature \citep[e.g.,][]{neuhaus1993conditional, janssen1997studentized, janssen1999testing, janssen2003bootstrap, chung2013exact,  pauly2015asymptotic}.

Clearly, replacing the $F$ statistic by the $X^2$ statistic does not affect the validity of the \FRT for testing $\HF$, because we always have an exact test for $\HF$ no matter which test statistic we use. Moreover, we derive a new result showing that the \FRT using $X^2$ as the test statistic can also control the asymptotic type I error for testing $\HN$. This means that the \FRT using $X^2$ as the test statistic can control the type I error under both $\HF$ and $\HN$ asymptotically, making $X^2$ a more attractive choice than the classical $F$ statistic for conducting the Fisher randomization test. Below, we formally state this new result.

\begin{theorem}
\label{thm::chisq-neyman}
Under $\HF$, the asymptotic distribution of $X^2$ is $\chi^2_{J-1}$ as $N\rightarrow \infty$.
Under $\HN$, the asymptotic distribution of $X^2$ is stochastically dominated by $\chi^2_{J-1}$, i.e., for any constant $a>0$,
$ \lim_{N\rightarrow \infty} \pr(X^2 \ge a) \le \pr(\chi^2_{J-1} \ge a) .$
\end{theorem}

\begin{remark}
\label{remark::x2}
Under $\HF$, the randomization distribution of $\SSTre/s_\obs^2$ follows $\chi^2_{J-1}$ asymptotically as shown in the Supplementary Material. Under $\HN$, however, the asymptotic distribution of $\SSTre/s_\obs^2$ is not $\chi^2_{J-1}$, and the asymptotic distribution of $F$ is not $F_{N-J,J-1}$ as suggested by Corollary \ref{coro::unbalanced}. Fortunately, if we weight each treatment square by the inverse of the sample variance of the outcomes, the resulting $X^2$ statistic preserves the asymptotic $\chi^2_{J-1}$ randomization distribution under $\HF$, and has an asymptotic distribution that is stochastically dominated by $\chi^2_{J-1}$ under $\HN$. 
\end{remark}

Therefore, under $\HN$, the type I error of the \FRT using $X^2$ does not exceed the nominal level. Although we can perform the \FRT by enumerating or simulating from all possible realizations of the treatment assignment, Theorem \ref{thm::chisq-neyman} suggests that an asymptotic rejection rule against $\HF$ or $\HN$ is $X^2 > x_{1-\alpha}$, the $1-\alpha$ quantile of the $\chi^2_{J-1}$ distribution. Because the asymptotic distribution of $X^2$ under $\HN$ is stochastically dominated by $\chi^2_{J-1}$, its true $1-\alpha$ quantile is asymptotically smaller than $x_{1-\alpha}$, and the corresponding \FRT is conservative in the sense of having smaller type I error than the nominal level asymptotically.

\begin{remark}
This asymptotic conservativeness is not particular to our test statistic, but rather a feature of finite population inference \citep{neyman::1923, aronow::2014, imbens::2015book}. It distinguishes Theorem \ref{thm::chisq-neyman} from previous results in the permutation test literature \citep[e.g.,][]{chung2013exact,  pauly2015asymptotic}, where the conservativeness did not appear and the correlation between the potential outcomes played no role in the theory.
\end{remark}

The form of $X^2$ in \eqref{eq:X^2} suggests its difference from $F$ when the potential outcomes have different variances under different treatment levels. Otherwise we show that they are asymptotically equivalent in the following sense.

\begin{corollary}
\label{coro::equiv}
If $S_{\cdot}^2(1) = \cdots = S_{\cdot}^2(J)$, then $(J-1) F \asim X^2$.
\end{corollary}

Under treatment effect additivity in Definition \ref{def:strict_additivity}, the condition $S_{\cdot}^2(1) = \cdots = S_{\cdot}^2(J)$ holds, and the equivalence between $(J-1) F $ and $ X^2$ guarantees that the Fisher randomization tests using $F$ and $X^2$ have the same asymptotic type I error and power. However, Corollary \ref{coro::equiv} is a large-sample result, and we evaluate it in finite samples in the Supplementary Material.

\section{Simulation}
\label{sec::numerical}

\subsection{Type I error of the \FRT using $F$}
\label{subsec::f}
In this subsection, we use simulation to evaluate the finite sample performance of the \FRT using $F$ under $\HN$. We consider the following three cases, where $\mathcal{N}(\mu, \sigma^2)$ denotes a normal distribution with mean $\mu$ and variance $\sigma^2$. We choose significance level $0.05$ for all tests.

Case 1.
For balanced experiments with sample sizes $N=45$ and $N=120$, we generate potential outcomes under two cases: (1.1) $Y_i(1)\sim \mathcal{N}(0,1)$, $Y_i(2)\sim \mathcal{N}(0, 1.2^2)$, $Y_i(3)\sim \mathcal{N}(0, 1.5^2)$; and (1.2) $Y_i(1)\sim \mathcal{N}(0,1)$, $Y_i(2)\sim \mathcal{N}(0, 2^2)$, $Y_i(3)\sim \mathcal{N}(0, 3^2)$. These potential outcomes are independently generated, and standardized to have zero means.

Case 2.
For unbalanced experiments with sample sizes $(N_1, N_2, N_3) = (10,20,30)$ and $(N_1, N_2, N_3) = (20,30,50)$, we generate potential outcomes under two cases: (2.1) $Y_i(1)\sim \mathcal{N}(0,1)$, $Y_i(2) = 2Y_i(1)$, $Y_i(3) = 3Y_i(1)$; and (2.2) $Y_i(1)\sim \mathcal{N}(0,1)$, $Y_i(2) = 3Y_i(1)$, $Y_i(3) = 5Y_i(1)$. These potential outcomes are standardized to have zero means. In this case, $p_1<p_2<p_3$ and $S^2_{\cdot}(1) < S^2_{\cdot}(2) < S^2_{\cdot}(3).$

Case 3.
For unbalanced experiments with sample sizes $(N_1, N_2, N_3) = (30,20,10)$ and $(N_1, N_2, N_3) = (50,30,20)$, we generate potential outcomes under two cases: (3.1) $Y_i(1)\sim \mathcal{N}(0,1)$, $Y_i(2) = 2Y_i(1)$, $Y_i(3) = 3Y_i(1)$; and (3.2) $Y_i(1)\sim \mathcal{N}(0,1)$, $Y_i(2) = 3Y_i(1)$, $Y_i(3) = 5Y_i(1)$. These potential outcomes are standardized to have zero means. In this case, $p_1>p_2>p_3$ and $S^2_{\cdot}(1) <S^2_{\cdot}(2) < S^2_{\cdot}(3).$

%
%

Once generated, the potential outcomes are treated as fixed constants. Over $2000$ simulated randomizations, we calculate the observed outcomes, and then perform the Fisher randomization test using $F$ to approximate the $p$-values by $2000$ draws of the treatment assignment. The histograms of the $p$-values are shown in Figures \ref{fig:subfigure1}--\ref{fig:subfigure3} corresponding to cases 1--3 above. We also report the rejection rates associated with these cases along with their standard errors in the next few paragraphs. 

In Figure \ref{fig:subfigure1}, the \FRT using $F$ is conservative with $p$-values distributed towards $1$. With larger heterogeneity in the potential outcomes, the histograms of the $p$-values have larger masses near $1$.
For case (1.1), the rejection rates are $0.010$ and $0.018$, and for case (1.2), the rejection rates are $0.023$ and $0.016$, for sample sizes $N=45$ and $N=120$ respectively, with all Monte Carlo standard errors no larger than $0.003$.

In Figure \ref{fig:subfigure2}, the sample sizes under each treatment level are increasing in the variances of the potential outcomes. The \FRT using $F$ is  conservative with $p$-values distributed towards $1$. Similar to Figure \ref{fig:subfigure1}, with larger heterogeneity in the potential outcomes, the $p$-values
have larger masses near $1$.
For case (2.1), the rejection rates are $0.016$ and $0.014$, and for case (2.2), the rejection rates are $0.015$ and $0.011$, for sample sizes $N=45$ and $N=120$ respectively, with all Monte Carlo standard errors no larger than $0.003.$

In Figure \ref{fig:subfigure3}, the sample sizes under different treatment levels are decreasing in the variances of the potential outcomes. 
For case (3.1), the rejection rates are $0.133$ and $0.126$, and for case (3.2), the rejection rates are $0.189$ and $0.146$, for sample sizes $N=45$ and $N=120$ respectively, with all Monte Carlo standard errors no larger than $0.009.$
The \FRT using $F$ does not preserve correct type I error with $p$-values distributed towards $0$. With larger heterogeneity in the potential outcomes, the $p$-values
have larger masses near $0$.

These empirical findings agree with our theory in Section \ref{sec::Neyman}, i.e., if the sample sizes under different treatment levels are decreasing in the sample variances of the observed outcomes, then the \FRT using $F$ may not yield correct type I error under $\HN$.

\begin{figure}[t]
\centering
\subfigure[Balanced experiments, case 1]{%
\includegraphics[width= 0.31 \textwidth]{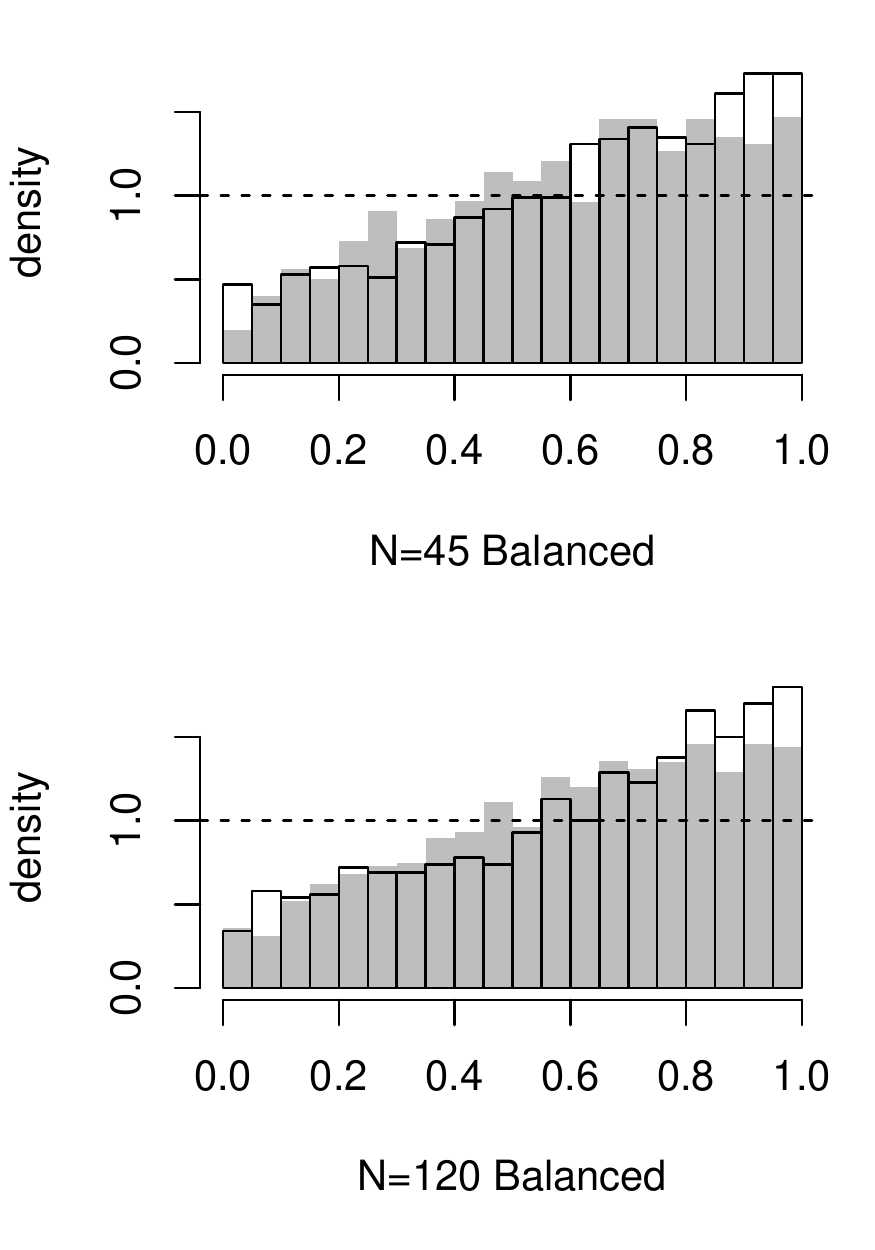}
\label{fig:subfigure1}}
\quad
\subfigure[Unbalanced experiments, case 2]{%
\includegraphics[width=0.31 \textwidth]{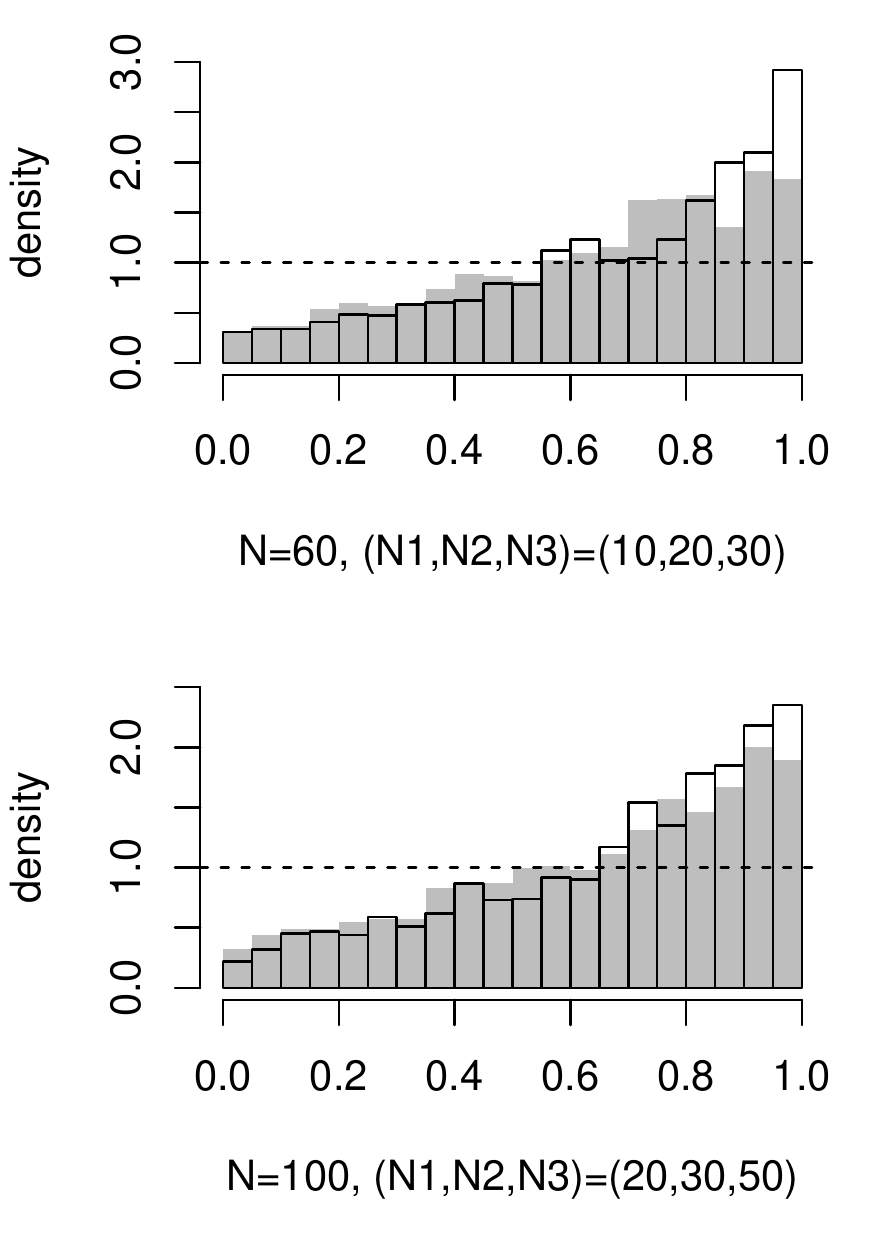}
\label{fig:subfigure2}}
\subfigure[Unbalanced experiments, case 3]{%
\includegraphics[width=0.31 \textwidth]{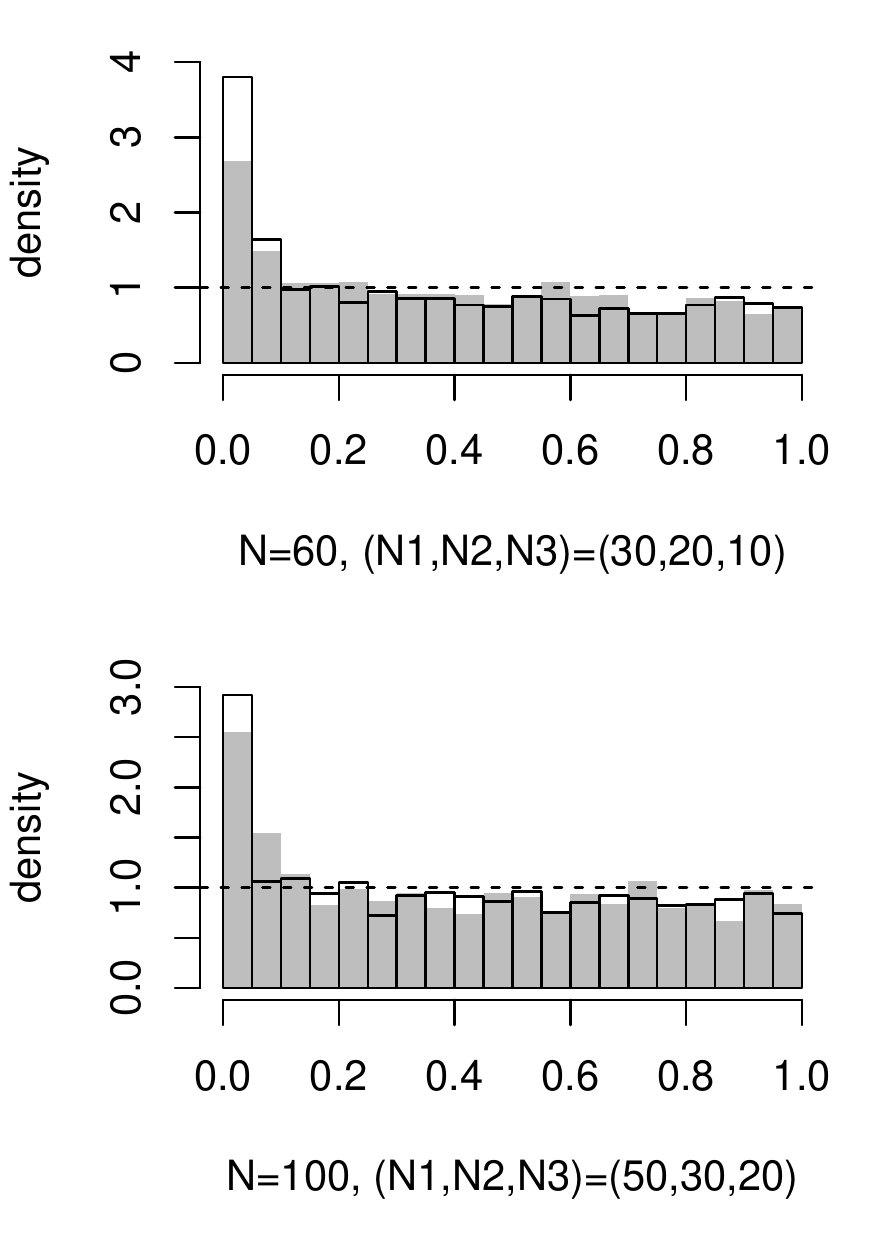}
\label{fig:subfigure3}}
\caption{Histograms of the $p$-values under $\HN$ based on the Fisher randomization tests using $F$, with grey histogram and white histograms for the first and second sub-cases.}
\label{fig:figure}
\end{figure}

\subsection{Type I error of the \FRT using $X^2$}
\label{subsec::x2}

Figure \ref{fig:figure_weight} shows the same simulation as Figure \ref{fig:figure}, but with test statistic $X^2$. 

Figure \ref{fig:subfigure1_weight} shows a similar pattern as Figure \ref{fig:subfigure1}. For case (1.1), the rejection rates are $0.016$ and $0.012$, and for case (1.2), the rejection rates are $0.014$ and $0.010$, for sample sizes $N=45$ and $N=120$ respectively, with all Monte Carlo standard errors no larger than $0.003$.

Figure \ref{fig:subfigure2_weight} shows better performance of the \FRT using $X^2$ than Figure \ref{fig:subfigure2}, with $p$-values closer to uniform. 
For case (2.1), the rejection rates are $0.032$ and $0.038$, and for case (2.2), the rejection rates are $0.026$ and $0.030$, for sample sizes $N=45$ and $N=120$ respectively, with all Monte Carlo standard errors no larger than $0.004.$

Figure \ref{fig:subfigure3_weight} shows much better performance of the \FRT using $X^2$ than Figure \ref{fig:subfigure3}, because the $p$-values are much closer to uniform. For case (3.1), the rejection rates are $0.052$ and $0.042$, and for case (3.2), the rejection rates are $0.048$ and $0.040$, for sample sizes $N=45$ and $N=120$ respectively, with all Monte Carlo standard errors no larger than $0.005.$
This agrees with our theory that the \FRT using $X^2$ can control the asymptotic type I error under $\HN$.


\begin{figure}[t]
\centering
\subfigure[Balanced experiments, case 1]{%
\includegraphics[width=0.31  \textwidth]{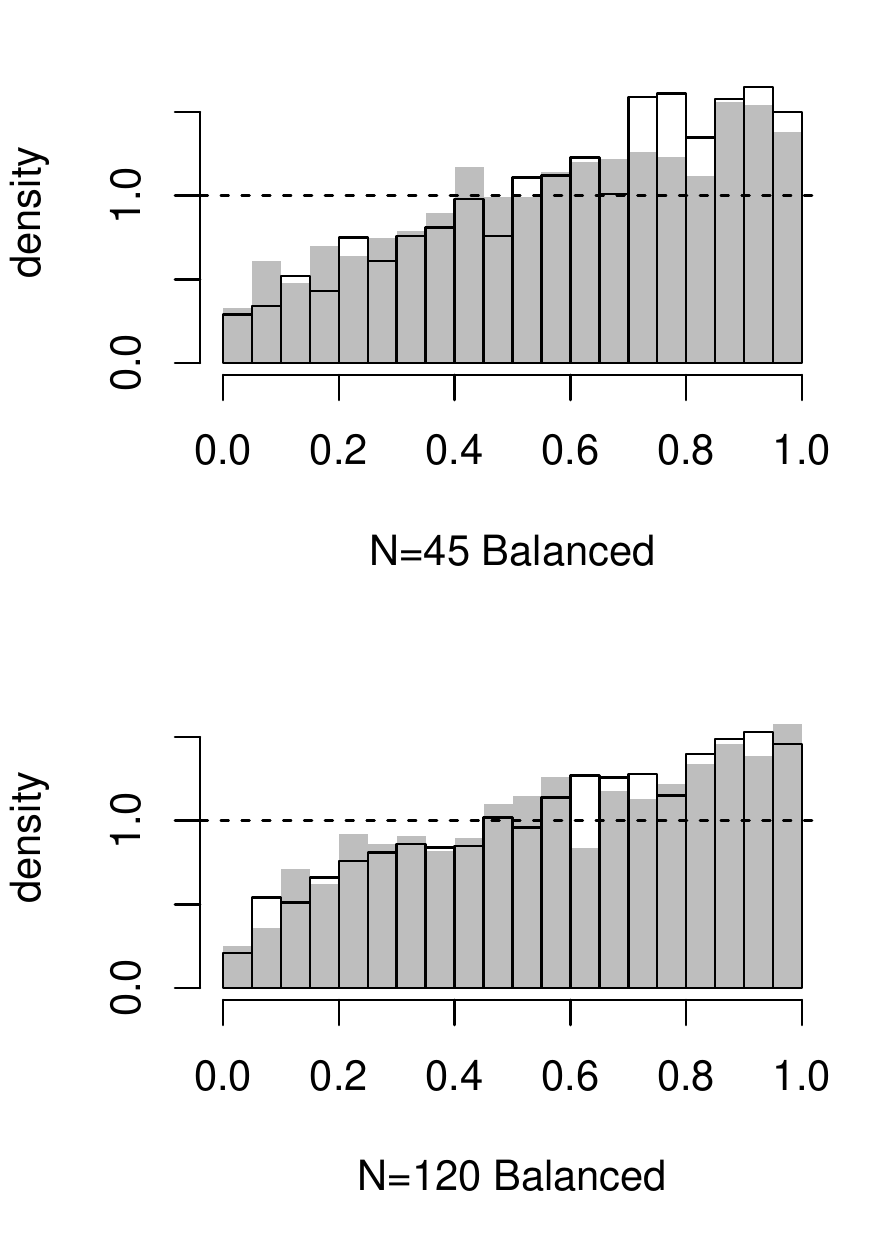}
\label{fig:subfigure1_weight}}
\quad
\subfigure[Unbalanced experiments, case 2]{%
\includegraphics[width=0.31 \textwidth]{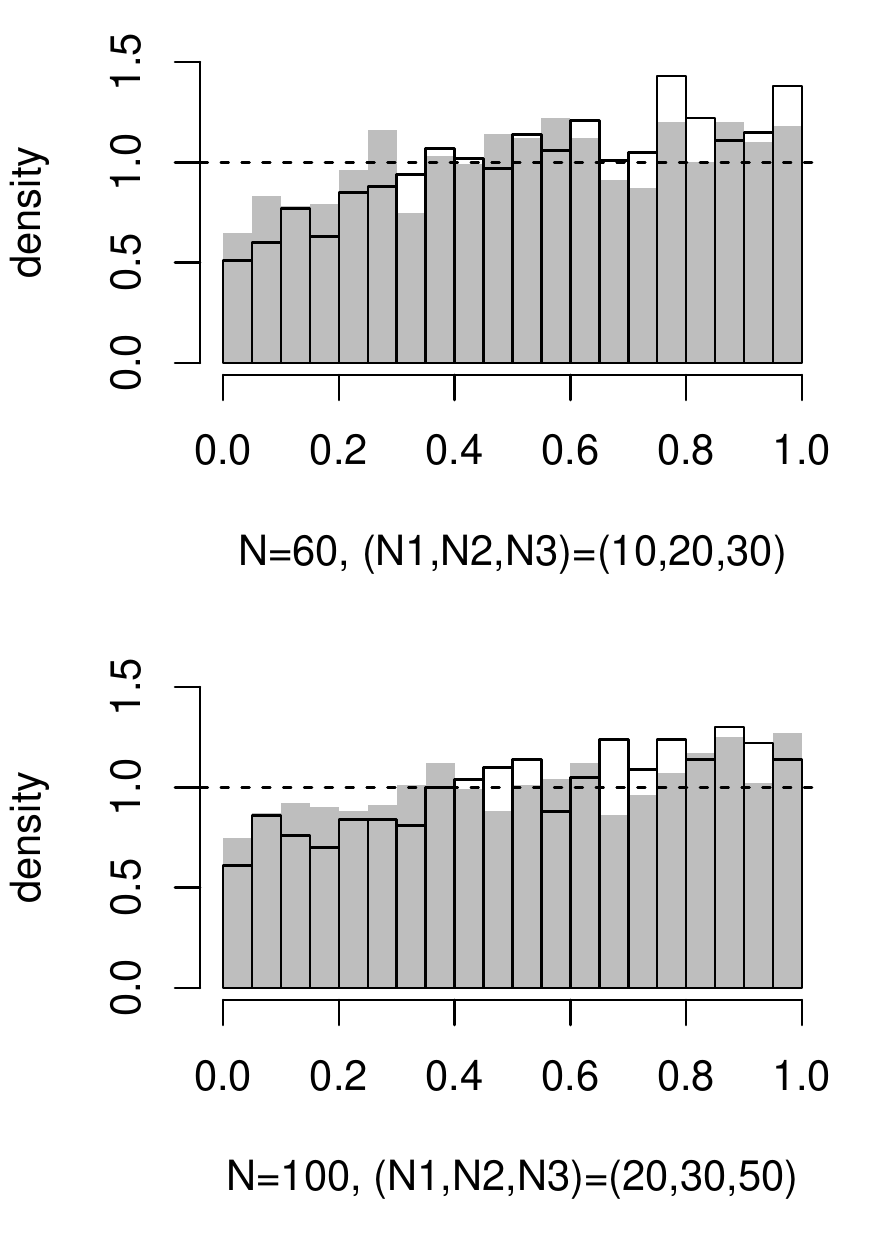}
\label{fig:subfigure2_weight}}
\subfigure[Unbalanced experiments, case 3]{%
\includegraphics[width=0.31  \textwidth]{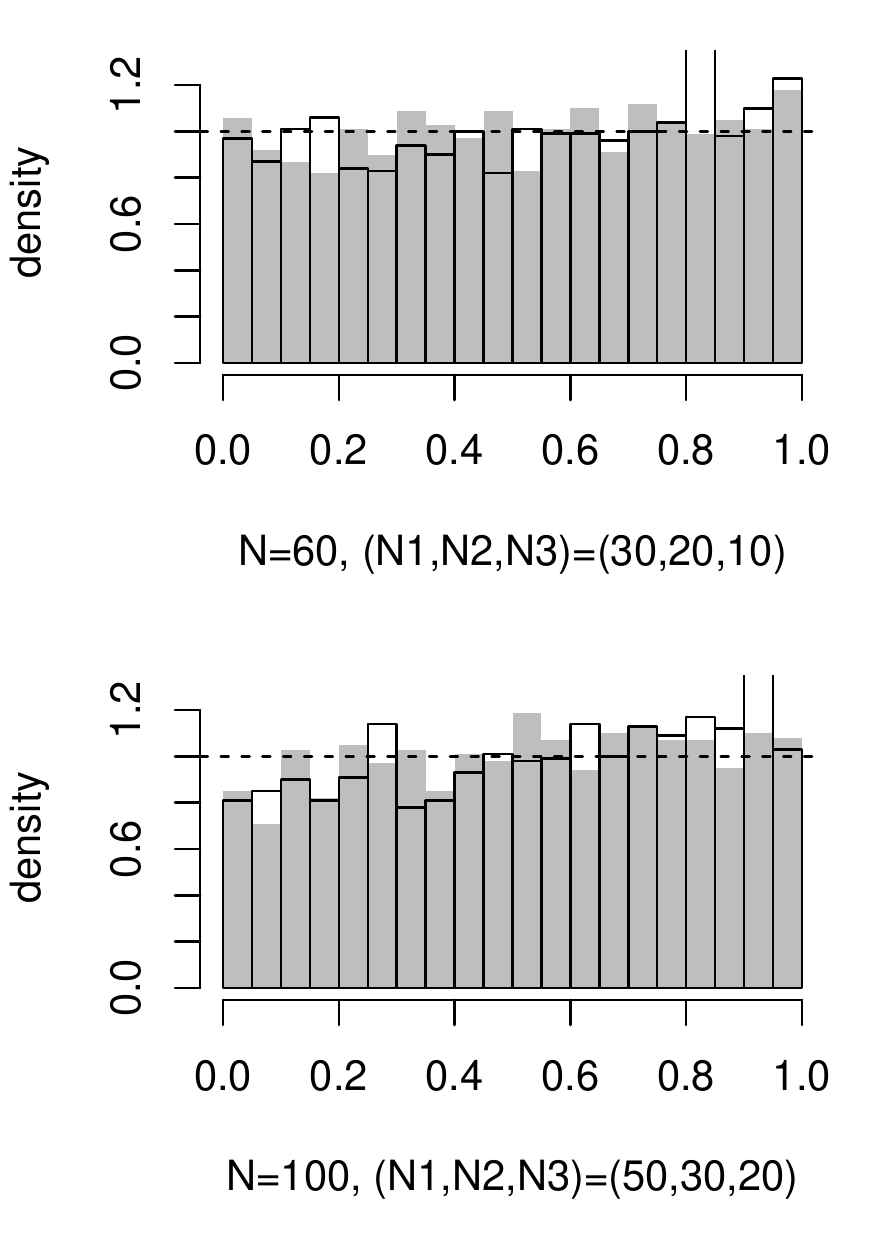}
\label{fig:subfigure3_weight}}
\caption{Histograms of the $p$-values under $\HN$ based on the Fisher randomization tests using $X^2$, with grey histogram and white histograms for the first and second sub-cases.}
\label{fig:figure_weight}
\end{figure}

\subsection{Power comparison of the Fisher randomization tests using $F$ and $X^2$}
\label{subsec::power}
In this subsection, we compare the powers of the Fisher randomization tests using $F$ and $X^2$ under alternative hypotheses. We consider the following cases.

Case 4.
For balanced experiments with sample sizes $N=30$ and $N=45$, we generate potential outcomes from $Y_i(1)\sim \mathcal{N}(0,1)$, $Y_i(2)\sim \mathcal{N}(0, 2^2)$, $Y_i(3)\sim \mathcal{N}(0, 3^2)$. These potential outcomes are independently generated, and shifted to have means $(0,1,2)$.

Case 5.
For unbalanced experiments with sample sizes $(N_1, N_2, N_3) = (10,20,30)$ and $(N_1, N_2, N_3) = (20,30,50)$, we first generate $Y_i(1)\sim \mathcal{N}(0,1)$ and standardize them to have mean zero, and we then generate $Y_i(2) = 3Y_i(1)+1$ and $Y_i(3) = 5Y_i(1) + 2$. In this case, $p_1<p_2<p_3$ and $S^2_{\cdot}(1) < S^2_{\cdot}(2) < S^2_{\cdot}(3).$

Case 6.
For unbalanced experiments with sample sizes $(N_1, N_2, N_3) = (30,20,10)$ and $(N_1, N_2, N_3) = (50,30,20)$, we generate potential outcomes the same as the above case 5. In this case, $p_1>p_2>p_3$ and $S^2_{\cdot}(1) <S^2_{\cdot}(2) < S^2_{\cdot}(3).$

Over $2000$ simulated data sets, we perform the \FRT using $F$ and $X^2$ and obtain the $p$-values by $2000$ draws of the treatment assignment. The histograms of the $p$-values, in Figures \ref{fig:subfigure1_power}--\ref{fig:subfigure3_power}, correspond to cases 4--6 above. The Monte Carlo standard errors for the rejection rates below are all close but no larger than $0.011.$

For case 4, the rejection rates using $X^2$ and $F$ are $0.290$ and $0.376$ respectively with sample size $N=30$, and $0.576$ and $0.692$ respectively with sample size $N=45$. For case 5, the powers using $X^2$ and $F$ are $0.178$ and $0.634$ respectively with sample size $N=60$, and $0.288$ and $0.794$ respectively with sample size $N=100$. Therefore, when the experiments are balanced or when the sample sizes are positively associated with the variances of the potential outcomes, the \FRT using $F$ has larger power than that using $X^2$.  

For case 6, the rejection rates using $X^2$ and $F$ are $0.494$ and $0.355$ respectively with sample size $N=60$, and $0.642$ and $0.576$ respectively with sample size $N=100$. Therefore, when the sample sizes are negatively associated with the variances of the potential outcomes, the \FRT using $F$ has smaller power than that using $X^2$.


\begin{figure}[t]
\centering
\subfigure[Balanced experiments, case 4]{%
\includegraphics[width=0.31 \textwidth]{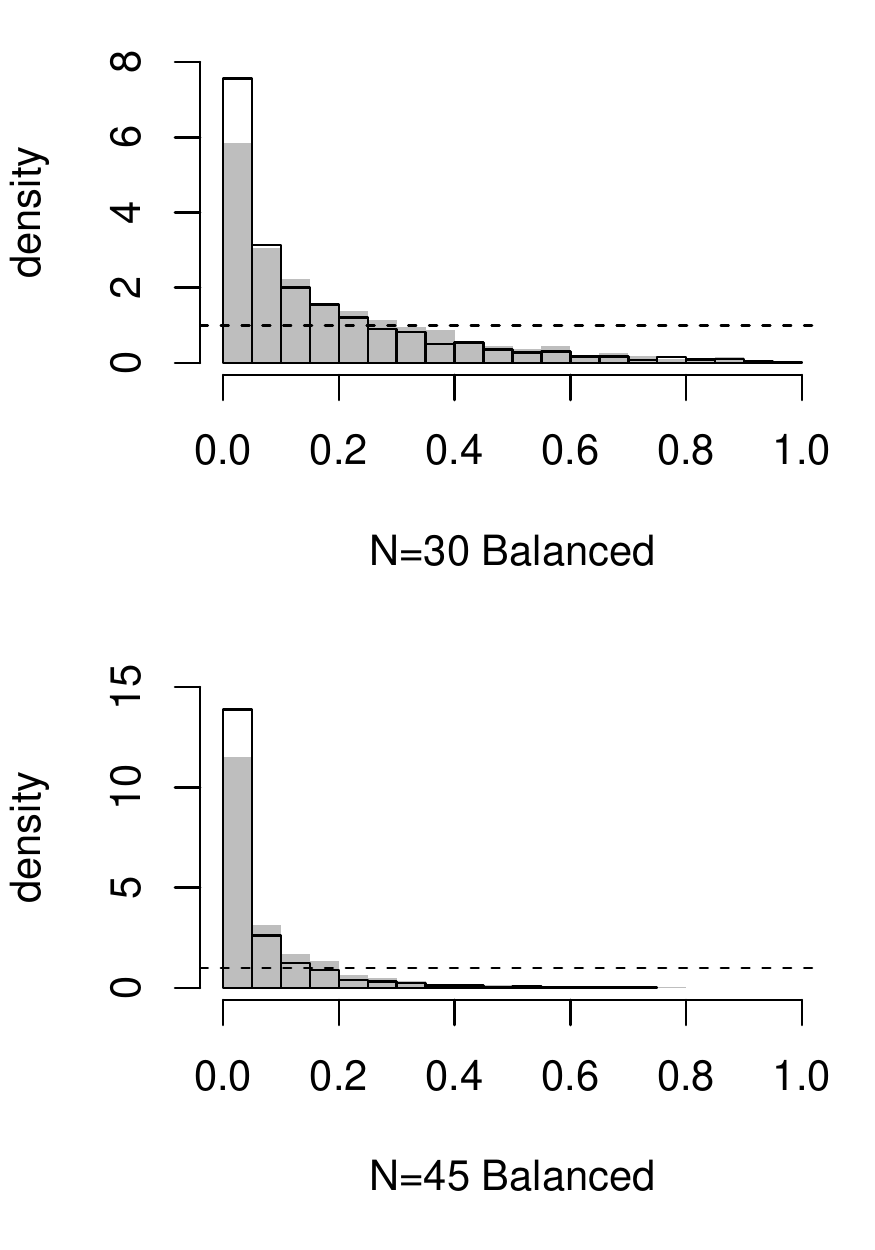}
\label{fig:subfigure1_power}}
\quad
\subfigure[Unbalanced experiments, case 5]{%
\includegraphics[width=0.31 \textwidth]{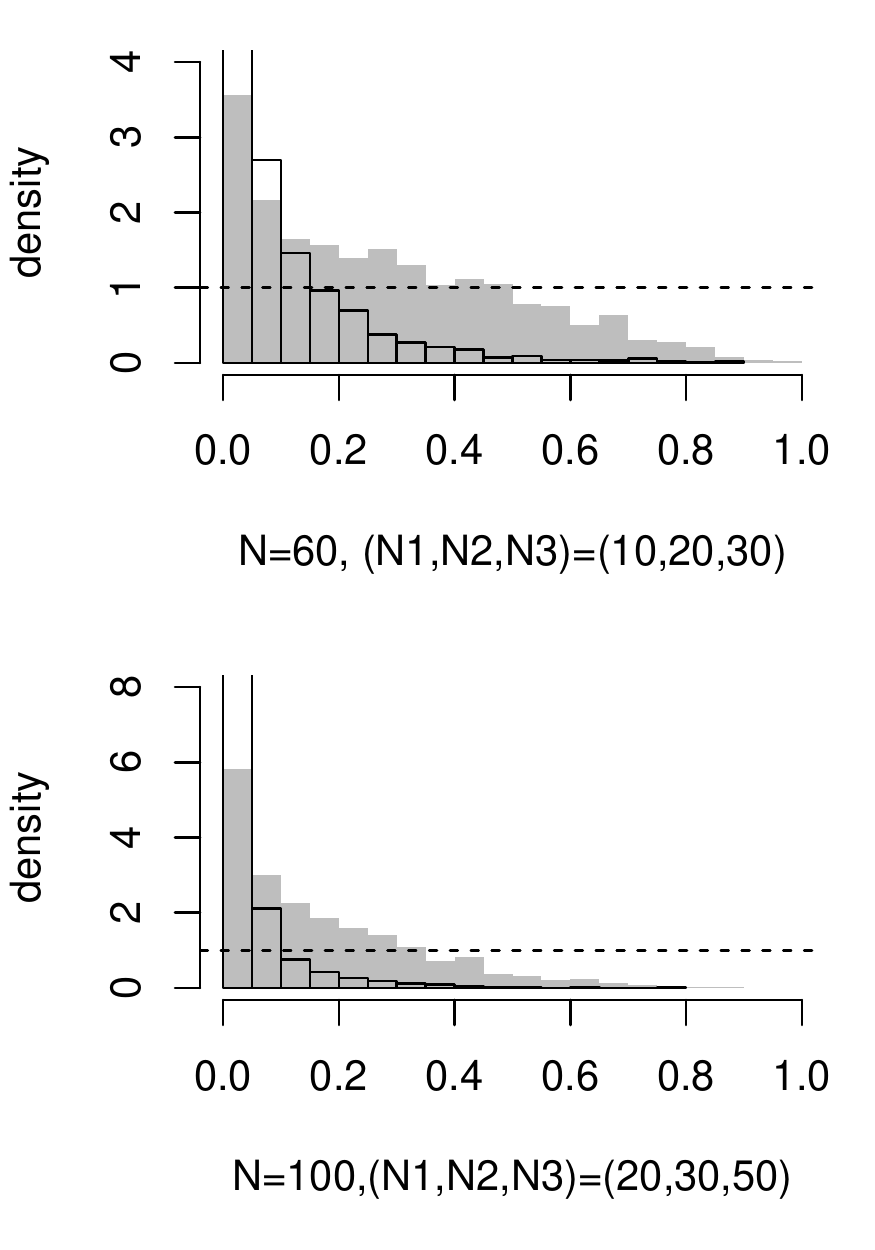}
\label{fig:subfigure2_power}}
\subfigure[Unbalanced experiments, case 6]{%
\includegraphics[width=0.31  \textwidth]{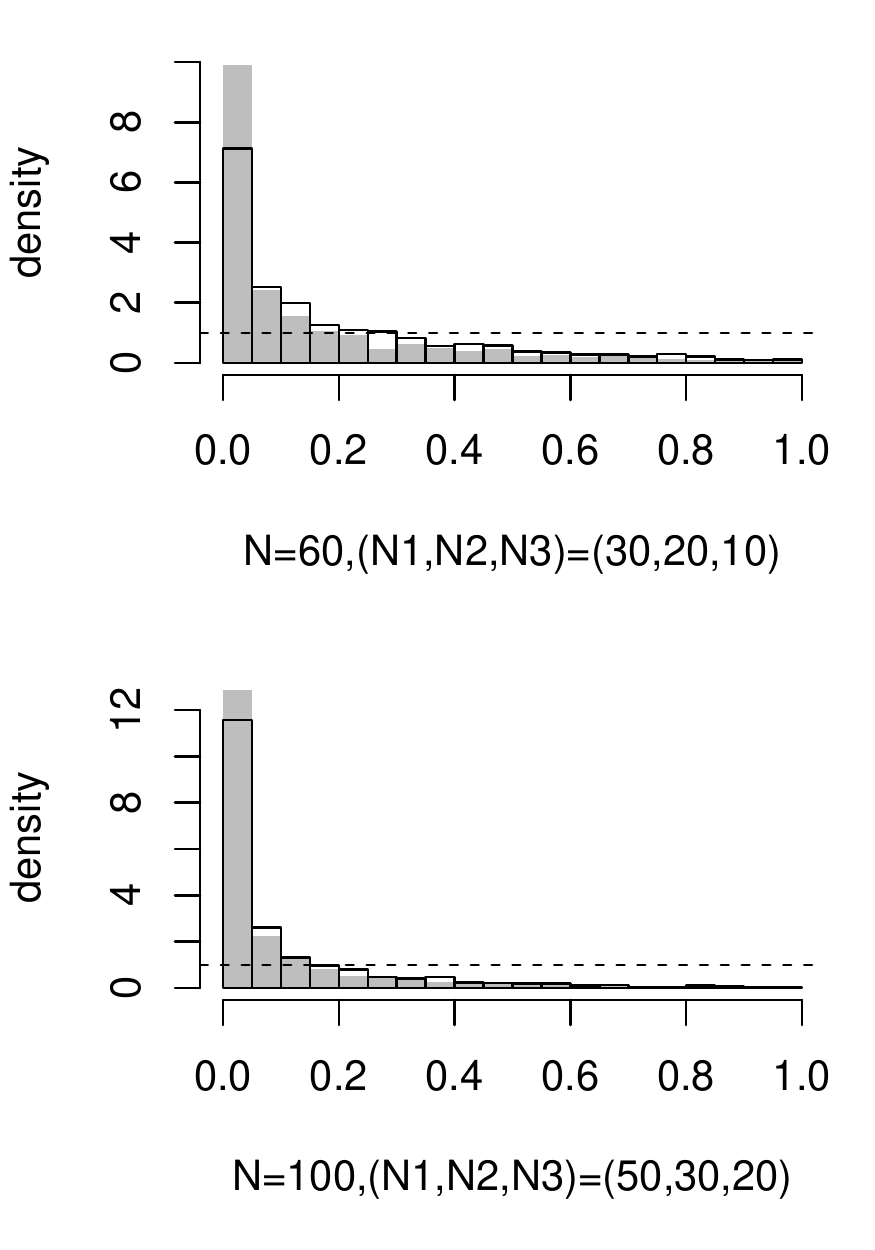}
\label{fig:subfigure3_power}}
\caption{Histograms of the $p$-values under alternative hypotheses based on the Fisher randomization tests using $F$ and $X^2$, with grey histograms for $X^2$ and white histograms for $F$.}
\label{fig:figure_power}
\end{figure}

\subsection{Simulation studies under other distributions and practical suggestions}

In the Supplementary Material, we give more numerical examples. First, we conduct simulation studies in parallel with \S\S \ref{subsec::f}--\ref{subsec::power} with outcomes generated from exponential distributions. The conclusions are nearly identical to those in \S\S \ref{subsec::f}--\ref{subsec::power}, because the finite population central limit theorems holds under mild moment conditions without imposing any distributional assumptions. 

Second, we use two numerical examples to illustrate the conservativeness issue in Theorem \ref{thm::chisq-neyman}. Third, we compare different behaviors of the Fisher randomization tests using $F$ and $X^2$ in two real-life examples.

\section{Discussion}

As shown in the proofs of Theorems \ref{thm::fisher-sharp-null} and \ref{thm::chisq-neyman} in the Supplementary Material, we need to analyze the eigenvalues of the covariance matrix of $\{ \bar{Y}_{\cdot}^\obs(1),  \ldots,  \bar{Y}_{\cdot}^\obs(J) \}$ to obtain the properties of $F$ and $X^2$ for general $J> 2$. Moreover, we consider the case with $J=2$ to gain more insights and to make connections with existing literature. For $j\neq j'$, an unbiased estimator for $\tau(j,j')$ is $\hat{\tau}(j,j') = \bar{Y}_{\cdot}^\obs(j) - \bar{Y}_{\cdot}^\obs(j') $, which has sampling variance $\var\{\hat{\tau}(j,j') \} = S_{\cdot}^2(j) / N_j  +  S_{\cdot}^2(j') / N_{j'}  -  S_{\cdot}^2(j\- j') / (N_j + N_{j'} )$ and an conservative variance estimator $s_\obs^2(j)/N_j +   s_\obs^2(j') / N_{j'}$ \citep{neyman::1923}.

\begin{corollary}
\label{coro::two-treatments}
When $J=2$, the $F$ and $X^2$ statistics reduce to 
$$
F\approx  \frac{     \hat{\tau}^2(1,2)     }{   s_\obs^2(1)/N_2 +   s_\obs^2(2) / N_1   },\quad
X^2 =  \frac{     \hat{\tau}^2(1,2)     }{   s_\obs^2(1)/N_1 +   s_\obs^2(2) / N_2   },
$$
where the approximation of $F$ is due to ignoring the difference between $N$ and $N-2$ and the difference between $N_j$ and $N_j-1$ $(j=1,2)$. Under $\HF$, $F\asim \chi^2_1$ and $X^2\asim \chi^2_1$. Under $\HN$, $F\asim C_1 \chi^2_1$ and $X^2\asim C_2 \chi^2_1$, where
\begin{eqnarray}
\label{eq::constants}
C_1 = \lim_{N\rightarrow + \infty}   \frac{    \var\{\hat{\tau}(1,2) \}    }
{  S_{\cdot}^2(1)/N_2 +   S_{\cdot}^2(2) / N_1       }, \quad
C_2 = \lim_{N\rightarrow + \infty}   \frac{    \var\{\hat{\tau}(1,2) \}    }
{  S_{\cdot}^2(1)/N_1 +   S_{\cdot}^2(2) / N_2       }  \leq 1.
\end{eqnarray}
\end{corollary}

Depending on the sample sizes and the finite population variances, $C_1$ can be either larger than or smaller than $1$. Consequently, using $F$ in the \FRT can be conservative or anti-conservative under $\HN$. In contrast, $C_2$ is always no larger than $1$, and therefore using $X^2$ in the \FRT is conservative for testing $\HN$. \citet{neyman::1923} proposed to use the square root of $X^2$ to test $\HN$ based on a normal approximation, which is asymptotically equivalent to the \FRT using $X^2$. Both are conservative unless the unit-level treatments are constant.

In practice, for treatment-control experiments, the difference-in-means statistic $\hat{\tau}(1,2)$ was widely used in the \FRT \citep{imbens::2015book}, which, however, can be conservative or anti-conservative for testing $\HN$ as shown in \citet{gail1996design}, \citet{lin2017placement} and \citet{ding2017paradox} using numerical examples. We formally state this result below, recognizing the equivalence between $\hat{\tau}(1,2)$ and $F$ in a two-sided test.

\begin{corollary}
\label{coro::difference-in-means}
When $J=2$, the two-sided \FRT using $\hat{\tau}(1,2)$ is equivalent to using  
$$
T^2 = \frac{ \hat{\tau}^2(1,2)  }{  N s_\obs^2/(N_1N_2)    }    \approx  \frac{\hat{\tau}^2(1,2) }{     s_\obs^2(1)/N_2 +   s_\obs^2(2) / N_1 + \hat{\tau}^2(1,2)/N   },
$$
where the approximation is due to ignoring the difference between $(N,N_1-1,N_2-1)$ and $(N,N_1,N_2)$. Under $\HF$, $T^2\asim F\asim \chi^2_1$, and under $\HN$, $T^2\asim F \asim C_1\chi^2_1$ with $C_1$ defined in \eqref{eq::constants}.   
\end{corollary}

\begin{remark}
Analogously, under the super population model, \citet{romano1990behavior} showed that the \FRT using $\hat{\tau}(1,2)$ can be conservative or anti-conservative for testing the hypothesis of equal means of two samples. \citet{janssen1997studentized, janssen1999testing} and \citet{chung2013exact} suggested using the studentized statistic, or equivalently $X^2$, to remedy the problem of possibly inflated type I error, which is asymptotically exact under the super population model. 
\end{remark}

After rejecting either $\HF$ or $\HN$, it is often of interest to test pairwise hypotheses, i.e., for $j\neq j'$, $Y_i(j) = Y_i(j')$ for all $i$, or $\bar{Y}_{\cdot}(j) = \bar{Y}_{\cdot}(j'). $ According to Corollaries \ref{coro::two-treatments} and \ref{coro::difference-in-means}, we recommend using the \FRT with test statistic
$
   \hat{\tau}^2(j,j')  /\{  s_\obs^2(j)/N_j +   s_\obs^2(j') / N_{j'}  \} ,
$
which will yield conservative type I error even if the experiment is unbalanced and the variances of the potential outcomes vary across treatment groups.

The analogue between our finite population theory and \citet{chung2013exact}'s super population theory suggests that similar results may also hold for layouts of higher order and other test statistics \citep{pauly2015asymptotic, chung2016asymptotically, chung2016multivariate, friedrich2017permuting}. In more complex experimental designs, often multiple effects are of interest simultaneously, raising the problem of multiple testings \citep{chung2016multivariate}. We leave these to future work.

\section*{Acknowledgement}
The research of Peng Ding was partially funded by Institute of Education Sciences, U.S.A. The authors thank Xinran Li, Zhichao Jiang, Lo-Hua Yuan and Robin Gong for suggestions for early versions of the paper. We thank a reviewer and the Associate Editor for helpful comments. 

%
%
%

\bibliographystyle{biometrika}
\bibliography{ACE}

\newpage

\renewcommand{\theequation}{S\arabic{equation}}
\renewcommand{\thelemma}{S\arabic{lemma}}
\renewcommand{\thesection}{S\arabic{section}}
\renewcommand{\thetheorem}{S\arabic{theorem}}
\renewcommand{\thefigure}{S\arabic{figure}}
\renewcommand{\thetable}{S\arabic{table}}
\renewcommand{\theexample}{S\arabic{example}}

\begin{center}
\Huge Supplementary Materials
\end{center}

\S \ref{sec::proofs} presents the proofs, \S \ref{sec::numerical} contains examples, and \S \ref{sec::morenumerical} gives additional simulation.

\section{Proofs}
\label{sec::proofs}

To prove the theorems, we need the following lemmas about completely randomized experiments.

\begin{lemma}\label{lemma::correlation-structure}
The treatment assignment indicator 
$W_{i}(j) $ is a Bernoulli random variable with mean $p_j = N_j/N$ and variance $p_j(1-p_j)$.
The covariances of the treatment assignment indicators are
$$
\begin{array}{lll}
\cov\{ W_i(j) , W_{i'} (j) \} &=  - p_j  (1-p_j) /(N-1),\quad  &(i\neq i') \\
\cov\{ W_i(j), W_i(j')  \}    &= -p_j p_{j'} ,\quad  &(j\neq j') \\
\cov\{  W_i(j), W_{i'}(j')  \} &= p_j p_{j'}/(N-1), \quad  &(i\neq i', j\neq j') .
\end{array}
$$
\end{lemma}

\begin{proof}[of Lemma \ref{lemma::correlation-structure}]
The proof is straightforward.
\end{proof}

\begin{lemma}\label{lemma:covariance}
Assume $(c_1, \ldots, c_N)$ and $(d_1,\ldots, d_N)$ are two fixed vectors with means $\bar{c}$ and $\bar{d}$, finite population variances $S_c^2$ and $S_d^2$. The finite population covariance is $S_{cd} = (S_c^2 + S_d^2 - S^2_{c\-d} ) /2 $, where $S^2_{c\-d}$ is the finite population variance of $(c_1-d_1, \ldots, c_N-d_N)$. For $j\neq j'$,
$$
\var \left\{  \frac{1}{N_j} \sumN W_i(j) c_i    \right\}   =  \frac{ 1-p_j }{N_j}  S_c^2 ,\quad
\cov\left\{    \frac{1}{N_j} \sumN W_i(j) c_i , \frac{1}{N_{j'}} \sumN W_i(j') d_i    \right \}  = - \frac{  S_{cd} }{N} .
$$
\end{lemma}

\begin{proof}
[of Lemma \ref{lemma:covariance}] 
Lemma \ref{lemma:covariance} is known, and its special forms appeared in \citet{kempthorne1955randomization}. We give an elementary proof for completeness. 
Applying Lemma \ref{lemma::correlation-structure}, we have
\begin{eqnarray*}
&&\var \left\{  \frac{1}{N_j} \sumN W_i(j) c_i    \right\}  \\
&=&   \frac{1}{N_j^2}  \var  \left\{   \sumN  W_i(j) (c_i - \bar{c} )    \right\} \\
&=& \frac{1}{N_j^2}  \left\{   \sumN \var\{W_i(j)  \} (c_i-\bar{c} ) ^2
-\mathop{\sum\sum}_{i\neq i'}  \cov\{ W_i(j), W_{i'}(j) \}  (c_i-\bar{c})(c_{i'} - \bar{c}) \right\}  \\
&=&  \frac{1}{N_j^2}  \left\{   \sumN p_j(1-p_j) (c_i-\bar{c}) ^2
-\mathop{\sum\sum}_{i\neq i'}   \frac{ p_j  (1-p_j) } { N-1} (c_i-\bar{c})(c_{i'} - \bar{c}) \right\}  \\
&=& \frac{1}{N_j^2}  \left\{  p_j(1-p_j)   \sumN  (c_i-\bar{c}) ^2
+   \frac{ p_j  (1-p_j) } { N-1}  \sumN    (c_i-\bar{c}) ^2   \right\} \\
&=&   \frac{1-p_j}{N_j}   S_c^2.
\end{eqnarray*}
For $j\neq j'$, applying Lemma \ref{lemma::correlation-structure} again, we have
\begin{eqnarray*}
&&\cov\left\{    \frac{1}{N_j} \sumN W_i(j) c_i , \frac{1}{N_j'} \sumN W_i(j') d_i    \right \}  \\
&=& \frac{1}{N_j N_{j'}} \cov\left\{   \sumN W_i(j) (c_i - \bar{c} ) ,
\sumN W_i(j') ( d_i -\bar{d})   \right\} \\
&=&   \frac{1}{N_j N_{j'}}  \left\{   \sumN \cov\{ W_i(j), W_i(j')\} (c_i - \bar{c} ) ( d_i -\bar{d}) \right. \\
&& ~~~~~~~~~~~~~\left. +
\mathop{\sum\sum}_{i\neq i'} \cov\{ W_i(j), W_{i'}(j')\}  (c_i - \bar{c} ) ( d_{i'} -\bar{d})   \right\} \\
&=&    \frac{1}{N_j N_{j'}}  \left\{  - \sumN p_j p_{j'}  (c_i - \bar{c} ) ( d_i -\bar{d}) +
\mathop{\sum\sum}_{i\neq i'} \frac{ p_j p_{j'}}{N-1}  (c_i - \bar{c} ) ( d_{i'} -\bar{d})    \right\} \\
&=& -  \frac{1}{N_j N_{j'}}  \left\{  p_j p_{j'}   \sumN  (c_i - \bar{c} ) ( d_i -\bar{d})   +
 \frac{ p_j p_{j'}}{N-1}   \sumN  (c_i - \bar{c} ) ( d_i -\bar{d})  \right\} \\
&=& - S_{cd}/N.
\end{eqnarray*}
\end{proof}

\begin{proof}
[of Theorem \ref{thm::fisher-sharp-null}]
Under $\HF$,  $\{ Y_i^\obs: i=1,\ldots, N  \}$ and $ \text{SSTot} = (N-1) s_{\obs}^2$ are fixed. Because $\{ Y_i^\obs: W_i(j) = 1  \}$ is a simple random sample from the finite population $\{ Y_i^\obs: i=1,\ldots, N  \}$, the sample mean $\bar{Y}_{\cdot}^\obs (j)$ is unbiased for the population mean $\bar{Y}_{\cdot}^\obs$, and the sample variance $ s^2_{\obs}(j) $ is unbiased for the population variance $ s_{\obs}^2$. Therefore,
$$
E( \SSRes)  = \sumJ  E \left\{  (N_j - 1) s^2_{\obs}(j) \right\}
= \sumJ  (N_j-1) s_{\obs}^2 = (N-J) s_{\obs}^2,
$$
which further implies that
$$
E( \SSTre) = \text{SSTot} - E( \SSRes) =  (N-1) s_{\obs}^2 - (N-J) s_{\obs}^2 = (J-1) s_{\obs}^2.
$$

Applying Lemma \ref{lemma:covariance}, we have
\begin{eqnarray}
\var\{\bar{Y}_{\cdot}^\obs(j)\} = \frac{1-p_j}{N_j}     s_{\obs}^2,\quad
\cov\{   \bar{Y}_{\cdot}^\obs(j) ,   \bar{Y}_{\cdot}^\obs(j' )    \} = - \frac{ s_{\obs}^2 }{ N  } . 
\label{eq::var-cov}
\end{eqnarray}
Therefore, the finite population central limit theorem \citep[][Theorem 5]{li2017general}, coupled with the variance and covariance formulae in \eqref{eq::var-cov}, implies
$$
V\equiv
\begin{bmatrix}
N_1^{1/2} \{  \bar{Y}_{\cdot}^\obs(1)  - \bar{Y}_{\cdot}^\obs \} \\
N_2^{1/2} \{  \bar{Y}_{\cdot}^\obs(2)  - \bar{Y}_{\cdot}^\obs \} \\
\vdots\\
N_J^{1/2} \{  \bar{Y}_{\cdot}^\obs(J)  - \bar{Y}_{\cdot}^\obs \}
\end{bmatrix}
\asim
\mathcal{N}_J\left[  0, s_{\obs}^2
\begin{pmatrix}
1-p_1 & -p_1^{1/2}p_2^{1/2} &\cdots& -p_1^{1/2} p_J^{1/2}\\
-p_2^{1/2} p_1^{1/2} & 1-p_2 &\cdots& -p_2^{1/2} p_J^{1/2}\\
\vdots&&&\vdots\\
-p_J^{1/2} p_1^{1/2}&-p_J^{1/2} p_2^{1/2}&\cdots& 1-p_J
\end{pmatrix}
\right] ,
$$
where $\mathcal{N}_J$ denotes a $J$-dimensional normal random vector. The above asymptotic covariance matrix can be simplified as $  s_\obs^2 ( I_J - qq ^\T  ) \equiv  s^2_\obs P $, where $I_J$ is the $J\times J$ identity matrix, and $q = (p_1^{1/2}, \ldots, p_J^{1/2})^\T $. The matrix $P$ is a projection matrix of rank $J-1$, which is orthogonal to the vector $q$. Consequently, the treatment sum of squares can be represented as
$
\SSTre = V^\T  V \asim  \chi_{J-1}^2 s_{\obs}^2,
$
and the F statistic can be represented as
\begin{eqnarray*}
F &=&  \frac{ \SSTre/(J-1) }{  \{ (N-1) s_{\obs}^2 - \SSTre\} / (N-J) }  \asim \frac{ \chi_{J-1}^2s_{\obs}^2 / (J-1)    }{  \{  (N-1)s_{\obs}^2 - \chi_{J-1}^2s_{\obs}^2  \} / (N-J) }  \\
&=&  \frac{  \chi^2_{J-1}/ (J-1) }{   \{  (N-1) - \chi^2_{J-1}  \}/(N-J) } \asim F_{J-1, N-J}  \asim \chi_{J-1}^2 / (J-1).
\end{eqnarray*}
\end{proof}

{\it Proof of Theorem \ref{thm::sampling-F-Neyman}.}
First, because $\bar{Y}_{\cdot}^\obs(j) = \sumN  W_i(j) Y_i(j)/N_j $, Lemma \ref{lemma:covariance} implies that $\bar{Y}_{\cdot}^\obs(j) $ has mean $\bar{Y}_{\cdot}(j)$ and variance $(1-p_j)S_{\cdot}^2(j)/ N_j$, and  
\begin{eqnarray*}
\cov \{\bar{Y}_{\cdot}^\obs(j) , \bar{Y}_{\cdot}^\obs(j') \}
&=&    \cov\left\{  \frac{1}{N_j }  \sumN  W_i(j) Y_i(j)   ,  \frac{1}{ N_{j'} } \sumN  W_i(j ' ) Y_i(j ' )     \right\} \\
&=&    - \frac{1}{2N} \{  S_{\cdot}^2(j) + S_{\cdot}^2(j') - S_{\cdot}^2(j\-j')  \} .
\end{eqnarray*}
Therefore,  
\begin{eqnarray*}
\var (\bar{Y}_{\cdot}^\obs)
&=& \sumJ  p_j^2  \var\{ \bar{Y}_{\cdot}^\obs(j) \}
+  \mathop{\sum\sum}_{j  \neq j'} p_j p_{j'} \cov\{  \bar{Y}_{\cdot}^\obs(j), \bar{Y}_{\cdot}^\obs(j')   \} \\
&=&  \sumJ  p_j^2 \frac{ 1-p_j} { N_j }S_{\cdot}^2(j)
-  \mathop{\sum\sum}_{j  \neq j'} p_j p_{j'}  \frac{1}{2N} \{  S_{\cdot}^2(j) + S_{\cdot}^2(j') - S_{\cdot}^2(j\-j')  \} \\
&=& \frac{1}{N}  \left\{    \sumJ  p_j  ( 1-p_j)  S_{\cdot}^2(j)   \right. \\
&& ~~~~~~~~~~ \left. - \frac{1}{2}        \mathop{\sum\sum}_{j  \neq j'} p_j p_{j'}   S_{\cdot}^2(j)
-  \frac{1}{2}        \mathop{\sum\sum}_{j  \neq j'} p_j p_{j'}   S_{\cdot}^2(j')
+ \frac{1}{2}        \mathop{\sum\sum}_{j  \neq j'} p_j p_{j'}   S_{\cdot}^2(j\-j')
\right\}.
\end{eqnarray*}
Because
\begin{eqnarray*}
 \mathop{\sum\sum}_{j  \neq j'} p_j p_{j'}   S_{\cdot}^2(j)  &=&  \sumJ  p_j  ( 1-p_j)  S_{\cdot}^2(j),\\
 \mathop{\sum\sum}_{j  \neq j'} p_j p_{j'}   S_{\cdot}^2(j') &=& \sumJ  p_{j'}  ( 1-p_{j'})  S_{\cdot}^2(j') = \sumJ  p_j  ( 1-p_j)  S_{\cdot}^2(j),
\end{eqnarray*}
the variance of $\bar{Y}_{\cdot}^\obs $ reduces to
$$
\var (\bar{Y}_{\cdot}^\obs)  =(2N)^{-1}  \mathop{\sum\sum}_{j  \neq j'} p_j p_{j'}   S_{\cdot}^2(j\-j') = \Delta / N.
$$

Second,  
\begin{eqnarray*}
\cov\{    \bar{Y}_{\cdot}^\obs(j), \bar{Y}_{\cdot}^\obs  \}&=&
p_j \var\{ \bar{Y}_{\cdot}^\obs(j)\}  +  \sum_{j'  \neq j}  p_{j'} \cov\{ \bar{Y}_{\cdot}^\obs(j), \bar{Y}_{\cdot}^\obs(j' )   \} \\
&=& \frac{1}{N}(1-p_j) S_{\cdot}^2(j) -   \frac{1}{2N}\sum_{j'  \neq j}  p_{j'}     \{  S_{\cdot}^2(j) + S_{\cdot}^2(j') - S_{\cdot}^2(j\-j')  \} .
\end{eqnarray*}
We further define  $\sum_{j'  \neq j}  p_{j'}   S_{\cdot}^2(j\-j') = \Delta_j.$
Because
$$
\sum_{j'  \neq j}  p_{j'}    S_{\cdot}^2(j) = (1-p_j)  S_{\cdot}^2(j) ,\quad
\sum_{j'  \neq j}  p_{j'}    S_{\cdot}^2(j') = S^2 - p_jS_{\cdot}^2(j),
$$
the covariance between $\bar{Y}_{\cdot}^\obs(j)$ and $\bar{Y}_{\cdot}^\obs$ reduces to
\begin{eqnarray*}
\cov\{    \bar{Y}_{\cdot}^\obs(j), \bar{Y}_{\cdot}^\obs  \}
&=&  (2N)^{-1}  \left\{    2(1-p_j) S_{\cdot}^2(j)  -  (1-p_j)  S_{\cdot}^2(j)   - S^2 +  p_jS_{\cdot}^2(j) +  \Delta_j \right\}  \\
&=&  (2N)^{-1}   \left\{ S_{\cdot}^2(j) - S^2 + \Delta_j \right\}.
\end{eqnarray*}

Third, $\bar{Y}_{\cdot}^\obs(j) -  \bar{Y}_{\cdot}^\obs $ has mean $ \bar{Y}_{\cdot}(j)  - \sumJ p_j \bar{Y}_{\cdot}(j)$ and variance
\begin{eqnarray*}
\var \{\bar{Y}_{\cdot}^\obs(j) -  \bar{Y}_{\cdot}^\obs\} &=&
\var\{\bar{Y}_{\cdot}^\obs(j) \} + \var ( \bar{Y}_{\cdot}^\obs ) - 2\cov\{  \bar{Y}_{\cdot}^\obs(j), \bar{Y}_{\cdot}^\obs   \}   \\
&=& \frac{1}{N}   \left\{
\frac{1-p_j}{p_j}  S_{\cdot}^2(j)  + \Delta - S_{\cdot}^2(j) + S^2 - \Delta_j
\right\}.
\end{eqnarray*}

Finally, the expectation of the treatment sum of squares is
\begin{eqnarray*}
E( \SSTre) &=& E \left[   \sumJ N_j    \{\bar{Y}_{\cdot}^\obs(j) -  \bar{Y}_{\cdot}^\obs\} ^2 \right] \\
&=&
\sumJ N_j  \left\{ \bar{Y}_{\cdot}(j)  - \sumJ p_j \bar{Y}_{\cdot}(j) \right\}^2+
 \sumJ p_j  \left\{
\frac{1-p_j}{p_j}  S_{\cdot}^2(j)  +  \Delta - S_{\cdot}^2(j) + S^2 - \Delta_j
\right\} ,
\end{eqnarray*}
which follows from the mean and variance formulas of $\bar{Y}_{\cdot}^\obs(j) -  \bar{Y}_{\cdot}^\obs$. Some algebra gives
\begin{eqnarray*}
E( \SSTre )&=&\sumJ N_j  \left\{ \bar{Y}_{\cdot}(j)  - \sumJ p_j \bar{Y}_{\cdot}(j) \right\}^2+
\sumJ (1-p_j) S_{\cdot}^2(j) +  \Delta - S^2 + S^2 - 2 \Delta \\
&=& \sumJ N_j  \left\{ \bar{Y}_{\cdot}(j)  - \sumJ p_j \bar{Y}_{\cdot}(j) \right\}^2+ \sumJ (1-p_j ) S_{\cdot}^2(j) - \Delta .
\end{eqnarray*}

Under $\HN$, i.e., $  \bar{Y}_{\cdot}(1) = \cdots  = \bar{Y}_{\cdot}(J)$, or, equivalently, $\bar{Y}_{\cdot}(j)  - \sumJ p_j \bar{Y}_{\cdot}(j)  = 0$ for all $j$, the expectation of the treatment sum of squares further reduces to
$$
E( \SSTre)  = \sumJ (1-p_j) S_{\cdot}^2(j) - \Delta . 
$$
Because $\{ Y_i^\obs: W_i(j) = 1  \}$ is a simple random sample from $\{ Y_i(j): i=1,2,\ldots, N  \}$, the sample variance is unbiased for the population variance, i.e.,
$
E \{  s_\obs^2(j)  \}
= S_{\cdot}^2(j).
$
Therefore, the mean of the residual sum of squares is
$$
E( \SSRes) = E  \left\{   (N_j - 1)   s_\obs^2(j)   \right\}
= \sumJ (N_j - 1)  S_{\cdot}^2(j) .    
$$
This completes the proof.  \hfill{$\square$} 

\begin{proof}
[of Corollary \ref{coro::1}]
Additivity implies $S^2 = S_{\cdot}^2(j)$ for all $j$ and $\Delta = 0$, and the conclusions follow.
\end{proof}

\begin{proof}
[of Corollary \ref{coro::2}]
For balanced designs, $p_j = 1/J, N_j  =N/J$ and $S^2 = \sumJ S_{\cdot}^2(j) / J$, and therefore Theorem 2 implies
\begin{eqnarray*}
E( \SSRes)&=&  \frac{N-J}{J}  \sumJ S_{\cdot}^2(j) = (N-J )S^2,\\
E( \SSTre) &=&   \frac{N}{J  } \sumJ  \{  \bar{Y}_{\cdot}(j) - \bar{Y}_{\cdot}(\cdot)  \}^2 +   (J-1) S^2 - \Delta.
\end{eqnarray*}
Moreover, under $\HN$, $E( \SSRes)$ is unchanged, and $E(\SSTre) =  (J-1) S^2 - \Delta.$
Therefore, the expectation of the mean treatment squares is no larger than the expectation of the mean residual squares, because
$
E( \MSRes) - E( \MSTre)   = \Delta / (J-1) \geq 0.
$
\end{proof}

\begin{proof}
[of Corollary \ref{coro::unbalanced}]
Under $\HN$,
\begin{eqnarray*}
E( \MSRes) - E( \MSTre)
&=& \sumJ  \left( \frac{N_j - 1}{ N- J} - \frac{1-p_j}{J - 1} \right)  S_{\cdot}^2(j) + \frac{\Delta}{J - 1} \\
&=& \frac{  (N-1)J }{  (J-1)(N-J) } \sumJ (p_j - J^{-1})S_{\cdot}^2(j) +  \frac{ \Delta} { J-1} .
\end{eqnarray*}
\end{proof}

To prove Theorem \ref{thm::chisq-neyman}, we need the following two lemmas: the first is about the quadratic form of the multivariate normal distribution, and the second, due to \citet{schur1911bemerkungen}, provides an upper bound for the largest eigenvalue of the element-wise product of two matrices. The proof of the first follows from straightforward linear algebra, and the proof of the second can be found in \citet[][Corollary 3]{styan1973hadamard}. Below we use $A*B$ to denote the element-wise product of $A$ and $B$, i.e, the $(i,j)$-th element of $A*B$ is the product of the $(i,j)$-th elements of $A$ and $B$, $A_{ij}B_{ij}.$

\begin{lemma}
\label{lemma::mvn}
If $X\sim \mathcal{N}_J(0, A)$, then $X^\T  B X \sim \sumJ \lambda_j  \xi_j$, where the $\xi_j$'s are iid $\chi_1^2$, and the $\lambda_j$'s are eigenvalues of $BA.$ 
\end{lemma}

\begin{lemma}
\label{lemma::schur}
If $A$ is positive semidefinite and $B$ is a correlation matrix, then the maximum eigenvalue of $A*B$ does not exceed the maximum eigenvalue of $A.$
\end{lemma}

\begin{proof}
[of Theorem \ref{thm::chisq-neyman}]
We first prove the result under $\HN$, and then view the result under $\HF$ as a special case. 

Let $Q_j = N_j/S_{\cdot}^2(j)$ for $j=1,\ldots,J$, and $Q = \sumJ Q_j$ be their sum. Define $q_w^\T = (Q_1^{1/2}, \ldots, Q_J^{1/2})/ Q^{1/2}$, and $P_w=  I_J -  q_w q_w^\T$ is a projection matrix of rank $J-1.$ Let $\bar{Y}^\obs_{w0} = Q^{-1} \sumJ  Q_j \bar{Y}_{\cdot}^\obs(j)$ be a weighted average of the means of the observed outcomes. According to \citet[][Proposition 3]{li2017general}, $s_\obs^2(j) - S_{\cdot}^2 (j) \rightarrow 0$ in probability $(j=1,\ldots,J)$. 
By Slutsky's Theorem, $X^2$ has the same asymptotic distribution as
$$
X_0^2 = \sumJ   Q_j  \left\{      \bar{Y}_{\cdot}^\obs(j) -    \bar{Y}^\obs_{w0}   \right\}^2 .
$$
Define $\rho_{jk}$ as the finite population correlation coefficient of potential outcomes $\{Y_i(j)\}_{i=1}^N$ and $\{ Y_i(k) \}_{i=1}^N$, and $R$ as the corresponding correlation matrix with $(j, k)$-th element $\rho_{jk}$. 
The finite population central limit theorem \citep[][Theorem 5]{li2017general} implies
\begin{eqnarray*}
V_0 & \equiv&
\begin{bmatrix}
Q_1^{1/2}  \{  \bar{Y}_{\cdot}^\obs(1)  - \bar{Y}_{\cdot}(1) \} \\
Q_2^{1/2} \{  \bar{Y}_{\cdot}^\obs(2)  -\bar{Y}_{\cdot}(2) \} \\
\vdots\\
Q_J^{1/2}  \{  \bar{Y}_{\cdot}^\obs(J)  - \bar{Y}_{\cdot}(J) \}
\end{bmatrix} \\
&\asim&
\mathcal{N}_J\left[  0,
\begin{pmatrix}
1-p_1 & -p_1^{1/2}p_2^{1/2} \rho_{12} &\cdots& -p_1^{1/2} p_J^{1/2}  \rho_{1J}\\
-p_2^{1/2} p_1^{1/2}  \rho_{21} & 1-p_2 &\cdots& -p_2^{1/2} p_J^{1/2} \rho_{2J} \\
\vdots&&&\vdots\\
-p_J^{1/2} p_1^{1/2 } \rho_{J1}&-p_J^{1/2} p_2^{1/2} \rho_{J2}&\cdots& 1-p_J
\end{pmatrix}
=  P * R
\right] ,
\end{eqnarray*}
where $P = I_J -qq^\T $ is the projection matrix defined in the proof of Theorem \ref{thm::fisher-sharp-null}. In the above, the mean and covariance matrix of the random vector $V_0$ follow directly from Lemmas \ref{lemma::correlation-structure} and \ref{lemma:covariance}.


Under $\HN$ with $\bar{Y}_{\cdot}(1) = \cdots = \bar{Y}_{\cdot}(J)$, we can verify that
$$
X_0^2 = \sumJ Q_j  \{  \bar{Y}_{\cdot}^\obs(j)  - \bar{Y}_{\cdot}(j) \}^2  
-
\frac{1}{  Q  }\left[  \sumJ Q_j \{  \bar{Y}_{\cdot}^\obs(j)  - \bar{Y}_{\cdot}(j) \} \right]^2,
$$
which can be further rewritten as a quadratic form \citep[cf.][]{chung2013exact} 
$$
X_0^2 = V_0^\T   (   I_J -  q_w q_w ^\T    )  V_0 \equiv  V_0^\T  P_w V_0.
$$
According to Lemma \ref{lemma::mvn}, $X_0^2$ has asymptotic distribution $\sum_{j=1}^{J-1} \lambda_j \xi_j$, where the $\lambda_j$'s are the $J-1$ nonzero eigenvalues of $P_w(P*R)$. The summation is from $j=1$ to $J-1$ because $P_w(P*R)$ has rank at most $J-1$. The eigenvalues $(\lambda_1, \ldots, \lambda_{J-1} ) $  are all smaller than or equal to the largest eigenvalue of $P*R$, because $P_w$ is a projection matrix.
According to Lemma \ref{lemma::schur}, the maximum eigenvalue of the element-wise product $P*R$ is no larger than the maximum eigenvalue of $P$, which is $1$. Therefore, $X_0^2 \asim \sum_{j=1}^{J-1} \lambda_j \xi_j$, where $\lambda_j \leq 1$ for all $j.$ Because the $\chi^2_{J-1}$ can be represented as $\xi_1  + \cdots + \xi_{J-1}$, it is clear that the asymptotic distribution of $X_0^2$ is stochastically dominated by $\chi^2_{J-1}$.

When performing the Fisher randomization test, we treat all observed outcomes as fixed, and consequently, the randomization distribution is essentially the repeated sampling distribution of $X^2$ under $Y_i(1) = \cdots = Y_i(J) = Y_i^\obs$. This restricts $S^2_{\cdot}(j)$ to be constant, and the correlation coefficients between potential outcomes to be $1$. Correspondingly, $P_w = P, R = 1_J 1_J^\T $, and the asymptotic covariance matrix of $V_0$ is $P$. Applying Lemma \ref{lemma::mvn} again, we know that the asymptotic randomization distribution of $X^2$ is $\chi^2_{J-1}$, because $PP=P$ has $J-1$ nonzero eigenvalues and all of them are $1$. 

Mathematically, the randomization distribution under $\HF$ is the same as the permutation distribution. Therefore, applying \citet{chung2013exact} yields the same result for $X^2$ under $\HF$.
\end{proof}

\begin{proof}
[of Corollary \ref{coro::equiv}]
As shown in the proof of Theorem \ref{thm::chisq-neyman}, $X^2$ is asymptotically equivalent to $X_0^2$, and therefore we need only to show the equivalence between $(J-1) F$ and $X^2_0.$ If $S_{\cdot}^2(1) = \cdots = S_{\cdot}^2(J) = S^2$, then $\bar{Y}^\obs_{w0}  = \bar{Y}_{\cdot}^\obs$, and
$$
X_0^2 =  \frac{  \sum_{j=1}^J  \{   \bar{Y}_{\cdot}^\obs(j) - \bar{Y}_{\cdot}^\obs   \}^2   }{    S^2   } = \frac{\SSTre}{ S^2 }.
$$
Because $\MSRes  = \sum_{j=1}^J  (N_j - 1)  s_\obs^2(j) / (N-J) $ converges to $S^2$ in probability \citep[][Proposition 3]{li2017general}, Slutsky's Theorem implies 
$$
(J-1)F =  \frac{\SSTre}{ \MSRes } \asim  \frac{\SSTre}{ S^2 }.
$$
Therefore, $(J-1) F \asim X_0^2\asim  X^2$.
\end{proof}

\begin{proof}
[of Corollary \ref{coro::two-treatments}]
First, we discuss $F.$ Because $\bar{Y}_{\cdot}^\obs = p_1 \bar{Y}_{\cdot}^\obs(1) + p_2 \bar{Y}_{\cdot}^\obs(2)$, we have
$$
 \bar{Y}_{\cdot}^\obs(1) - \bar{Y}_{\cdot}^\obs  = p_2 \hat{\tau}(1,2),\qquad
 \bar{Y}_{\cdot}^\obs(2) - \bar{Y}_{\cdot}^\obs  = -p_1 \hat{\tau}(1,2).
$$
The treatment sum of squares reduces to
$$
\SSTre =  N_1 \left\{   \bar{Y}_{\cdot}^\obs(1) - \bar{Y}_{\cdot}^\obs  \right\}^2 
+ N_2  \left\{  \bar{Y}_{\cdot}^\obs(2) - \bar{Y}_{\cdot}^\obs  \right\}^2  = N p_1 p_2 \hat{\tau}^2(1,2),$$
and the residual sum of squares reduces to $\SSRes = (N_1-1)  s_\obs^2(1) + (N_2-1)s_\obs^2(2)$.
Therefore, the $F$ statistic reduces to
$$
F = \frac{  \SSTre }{  \SSRes / (N-2)} 
= \frac{  \hat{\tau}^2(1,2)  }{      \frac{N (N_1-1) }{ (N-2)N_1N_2  }   s_\obs^2(1) + \frac{N(N_2-1)}{ (N-2)N_1 N_2 }   s_\obs^2(2)  }
\approx \frac{     \hat{\tau}^2(1,2)     }{   s_\obs^2(1)/N_2 +   s_\obs^2(2) / N_1   },
$$
where the approximation follows from ignoring the difference between $N$ and $N-2$ and the difference between $N_j$ and $N_j-1$ $(j=1,2)$. Following from Theorem \ref{thm::fisher-sharp-null} or proving it directly, we know that $F \asim F_{1,N-2} \asim \chi^2_1$ under $\HF$. However, under $\HN$, \citet{neyman::1923}, coupled with the finite population central limit theorem \citep[][Theorem 5]{li2017general}, imply
$$
\frac{  \hat{\tau}(1,2)  }{   \left\{   \frac{ S_{\cdot}^2(1)}{N_1}  + \frac{S_{\cdot}^2(2)}{N_2} - \frac{S_{\cdot}^2(1\- 2)}{N}  \right\}^{1/2}     } \asim \mathcal{N}(0,1),
$$ 
and $ s_\obs^2(j) \rightarrow S_{\cdot}^2(j)$ in probability $(j=1,2)$. Therefore, the asymptotic distribution of $F$ under $\HN$ is
$
F \asim C_1 \chi^2_1,
$
where 
$$
C_1 = \lim_{N\rightarrow + \infty}   \frac{     S_{\cdot}^2(1) / N_1  +  S_{\cdot}^2(2) / N_2  -  S_{\cdot}^2(1\- 2) / N    }
{  S_{\cdot}^2(1)/N_2 +   S_{\cdot}^2(2) / N_1       } .
$$

Second, we discuss $X^2$. Because
$$
\bar{Y}_w^\obs =  \left\{   \frac{N_1}{   s_\obs^2(1)  } \bar{Y}_{\cdot}^\obs(1)    + \frac{ N_2 }{   s_\obs^2(2)    }    \bar{Y}_{\cdot}^\obs(2)     \right\} \Big / 
\left\{   \frac{N_1}{   s_\obs^2(1)  } + \frac{ N_2 }{   s_\obs^2(2)    }      \right\} , 
$$
we have
\begin{eqnarray*}
\bar{Y}_{\cdot}^\obs(1) -  \bar{Y}_w^\obs  &=& \frac{N_2}{   s_\obs^2(2)   }   \hat{\tau}^2(1,2)  
\Big / 
\left\{   \frac{N_1}{   s_\obs^2(1)  } + \frac{ N_2 }{   s_\obs^2(2)    }      \right\}, \\
\bar{Y}_{\cdot}^\obs(2) -  \bar{Y}_w^\obs  &=&  - \frac{N_1}{   s_\obs^2(1)   }   \hat{\tau}^2(1,2)  
\Big / 
\left\{   \frac{N_1}{   s_\obs^2(1)  } + \frac{ N_2 }{   s_\obs^2(2)    }      \right\}.
\end{eqnarray*}
Therefore, the $X^2$ statistic reduces to
\begin{eqnarray*}
X^2&=& \left\{   \frac{N_1}{   s_\obs^2(1)  }  \frac{N_2^2}{   s_\obs^4(2)   }   \hat{\tau}^2(1,2)  
+  \frac{ N_2 }{   s_\obs^2(2)    }    \frac{N_1^2}{   s_\obs^4(1)  }   \hat{\tau}^2(1,2)     \right\}
\Big / 
\left\{   \frac{N_1}{   s_\obs^2(1)  } + \frac{ N_2 }{   s_\obs^2(2)    }      \right\}^2 \\
&=&  \frac{     \hat{\tau}^2(1,2)     }{   s_\obs^2(1)/N_1 +   s_\obs^2(2) / N_2   }.
\end{eqnarray*}
Following from Theorem \ref{thm::chisq-neyman} or proving it directly, we know that $X^2 \asim \chi^2_1 $ under $\HF$. However, under $\HN$, we can use an argument similar to that for $F$ and obtain
$
X^2 \asim C_2 \chi^2_1,
$
where
$$
C_2 = \lim_{N\rightarrow + \infty}   \frac{     S_{\cdot}^2(1) / N_1  +  S_{\cdot}^2(2) / N_2  -  S_{\cdot}^2(1\- 2) / N    }
{  S_{\cdot}^2(1)/N_1 +   S_{\cdot}^2(2) / N_2       }  \leq 1.
$$
The constant $C_2$ is smaller than or equal to $1$ with equality holding if the limit of $S_{\cdot}^2(1\- 2) $ is zero, i.e., the unit-level treatment effects are constant asymptotically.
\end{proof}

\begin{proof}
[of Corollary \ref{coro::difference-in-means}]
In the Fisher randomization test, $s_\obs^2$ is fixed, and therefore using $\hat{\tau}(1,2)$ is equivalent to using $T^2$. Using simple algebra similar to \citet{ding2017paradox}, we have the following decomposition
$$
(N - 1)s_{\obs}^2 = (N_1 - 1) s_\obs^2(1) + (N_2 - 1) s_\obs^2(2) + N_1N_2 \hat{\tau}(1,2)/N, 
$$
which implies the equivalent formula of $T^2$ in Corollary \ref{coro::difference-in-means}. Under $\HF$ or $\HN$, $\hat{\tau}(1,2) \rightarrow 0$ in probability, which coupled with Slutsky's Theorem, implies the asymptotic equivalence $T^2\asim F.$
\end{proof}

\section{Numerical Examples} 
\label{sec::numerical}

\begin{example}
\label{eg::example}
We consider $J=3$, sample sizes $N_1=120, N_2 = 80$ and $ N_3 = 40$. We generate the first set of potential outcomes from
\begin{eqnarray}
Y_i(1)\sim \mathcal{N}(0,1), Y_i(2)= 3Y_i(1), Y_i(3)=5Y_i(1), \label{eg::1}
\end{eqnarray}
and the second set of potential outcomes from
\begin{eqnarray}
Y_i(1)\sim \mathcal{N}(0,1), Y_i(2)\sim \mathcal{N}(0,3^2), Y_i(3)\sim \mathcal{N}(0,5^2). \label{eg::2}
\end{eqnarray}
After generating the potential outcomes, we center the $Y_i(j)$'s by subtracting the mean to make $\bar{Y}_{\cdot}(j)=0$ for all $j$ so that $\HN$ holds. Figure \ref{fg::sampling} shows the distributions of $X^2$ over repeated sampling of the treatment assignment vector $(W_1,\ldots, W_N)$ for potential outcomes generated from \eqref{eg::1} and \eqref{eg::2}. The true sampling distributions under both cases are stochastically dominated by $\chi^2_2$.
Under \eqref{eg::1}, the correlation coefficients between the potential outcomes are $1$; whereas under \eqref{eg::2}, the correlation coefficients are $0$.
With less correlated potential outcomes, the gap between the true distribution and $\chi^2_2$ becomes larger.

\begin{figure}[t]
\centering
\includegraphics[width = 0.7 \textwidth]{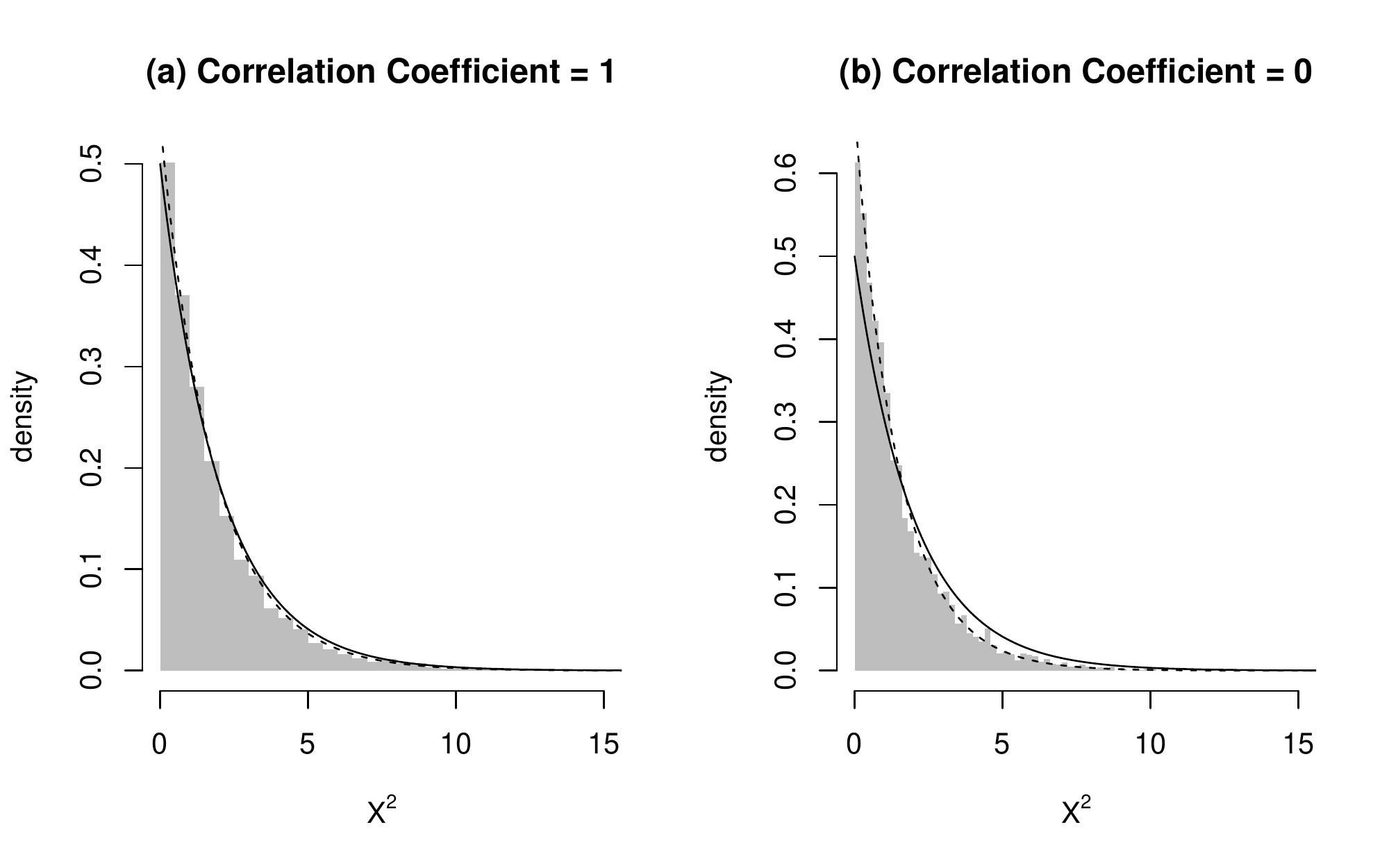}
\caption{Distributions of $X^2$. The histograms are the sampling distributions, the dotted lines are the asymptotic distributions, and the solid lines are the $\chi^2_{2}$ distribution. }\label{fg::sampling}
\end{figure}

\end{example}

\begin{example}
We use an example from \citet[][Exercise 3.15]{montgomery2000design} with $4$ treatment levels. The sample variances and the sample sizes differ for the four treatment levels, as shown in Table \ref{tb::monto}.
The $p$-values of the \FRT using $F$ and $X^2$ are
$0.003$ and $0.010$, respectively. If we choose a stringent size, say $\alpha=0.01$, then the evidence against the null is strong from the first test, but the evidence is weak from the second test. If our interest is $\HN$, then the different strength of evidence may be due to the different variances and sample sizes of the treatment groups. Because of this, we recommend making decision based on the \FRT using
$X^2$.

 \begin{table}[t]
\centering
\caption{A randomized experiment with $J=4$}\label{tb::monto}
\begin{tabular}{ccccc}
\hline
 & 1 & 2 & 3 & 4\\
 \hline
observed outcome
&58.2&56.3&50.1&52.9\\
&57.2&54.5&54.2&49.9\\
&58.4&57.0&55.4&50.0\\
&55.8&55.3&&51.7\\
&54.9&&&\\
\hline
sample size&5&4&3&4\\
mean&
56.9&55.8&53.2&51.1\\
variance&
2.3&1.2&7.7&2.1\\
\hline
\end{tabular}
\end{table}

\end{example}

\begin{example}
We reanalyze the data from \citet{angrist2009incentives}, which contain a control group and $3$ treatment groups designed to improve academic performance among college freshmen. Table \ref{tb::angrist} summaries the sample sizes, means and variances of the final grades under $4$ treatment groups. The $p$-values of the \FRT using $F$ and $X^2$ are
$0.058$ and $0.045$, respectively. The Fisher randomization tests using $F$ and $X^2$ give different conclusions at the commonly used significance level of $0.05$. In this unbalanced experiment, the \FRT using $F$ is less powerful.

\begin{table}[t]
\centering
\caption{A randomized experiment with $J=4$, where control, sfp, ssp and sfsp denote the four treatment groups.} \label{tb::angrist}
\begin{tabular}{rrrrr}
  \hline
 & control & sfp & ssp & sfsp \\ 
  \hline
 sample size & 854 & 219 & 212 & 119 \\ 
  mean & 63.86 & 65.83 & 64.13 & 66.10 \\ 
  variance & 144.97 & 124.45 & 159.76 & 114.33 \\ 
   \hline
\end{tabular}
\end{table}

\end{example}

\section{More Simulation With Nonnormal Outcomes} 
\label{sec::morenumerical}

\subsection{Type I error of the \FRT using $F$}
\label{subsec::f-S}
In this subsection, we use simulation to evaluate the finite sample performance of the \FRT using $F$ under $\HN$. We consider the following three cases, where $\mathcal{E}$ denotes an exponential distribution with mean $1$. 

Case S1.
For balanced experiments with sample sizes $N=45$ and $N=120$, we generate potential outcomes under two cases: (S1.1) $Y_i(1)\sim \mathcal{E} $, $Y_i(2)\sim \mathcal{E} / 0.7$, $Y_i(3)\sim \mathcal{E} / 0.5 $; and (S1.2) $Y_i(1)\sim \mathcal{E}$, $Y_i(2)\sim \mathcal{E} / 0.5$, $Y_i(3)\sim \mathcal{E} / 0.3$. These potential outcomes are independently generated, and standardized to have zero means.

Case S2.
For unbalanced experiments with sample sizes $(N_1, N_2, N_3) = (10,20,30)$ and $(N_1, N_2, N_3) = (20,30,50)$, we generate potential outcomes under two cases: (S2.1) $Y_i(1)\sim \mathcal{E}$, $Y_i(2) = 2Y_i(1)$, $Y_i(3) = 3Y_i(1)$; and (S2.2) $Y_i(1)\sim \mathcal{E}$, $Y_i(2) = 3Y_i(1)$, $Y_i(3) = 5Y_i(1)$. These potential outcomes are standardized to have zero means. In this case, $p_1<p_2<p_3$ and $S^2_{\cdot}(1) < S^2_{\cdot}(2) < S^2_{\cdot}(3).$

Case S3.
For unbalanced experiments with sample sizes $(N_1, N_2, N_3) = (30,20,10)$ and $(N_1, N_2, N_3) = (50,30,20)$, we generate potential outcomes under two cases: (S3.1) $Y_i(1)\sim \mathcal{E}$, $Y_i(2) = 1.2Y_i(1)$, $Y_i(3) = 1.5Y_i(1)$; and (S3.2) $Y_i(1)\sim \mathcal{E} $, $Y_i(2) = 1.5Y_i(1)$, $Y_i(3) = 2Y_i(1)$. These potential outcomes are standardized to have zero means. In this case, $p_1>p_2>p_3$ and $S^2_{\cdot}(1) <S^2_{\cdot}(2) < S^2_{\cdot}(3).$

We follow \S \ref{subsec::f} and obtain the same conclusions about the Fisher randomization test using $F$, because Figures \ref{fig:figure} and \ref{fig:figure-S} exhibit the same pattern.

In Figure \ref{fig:subfigure1-S}, for case (S1.1), the rejection rates are $0.022$ and $0.014$, and for case (S1.2), the rejection rates are $0.030$ and $0.030$, for sample sizes $N=45$ and $N=120$ respectively. 
In Figure \ref{fig:subfigure2-S}, for case (S2.1), the rejection rates are $0.018$ and $0.024$, and for case (2.2), the rejection rates are $0.026$ and $0.018$, for sample sizes $N=45$ and $N=120$ respectively. The Monte Carlo standard errors are all close to but no larger than $0.003.$

In Figure \ref{fig:subfigure3-S}, for case (S3.1), the rejection rates are $0.076$ and $0.086$, and for case (S3.2), the rejection rates are $0.108$ and $0.109$, for sample sizes $N=45$ and $N=120$ respectively, with all Monte Carlo standard errors no larger than $0.008.$
In these two cases, the \FRT using $F$ does not preserve correct type I error.

\begin{figure}[t]
\centering
\subfigure[Balanced experiments, case S1]{%
\includegraphics[width=0.31\textwidth]{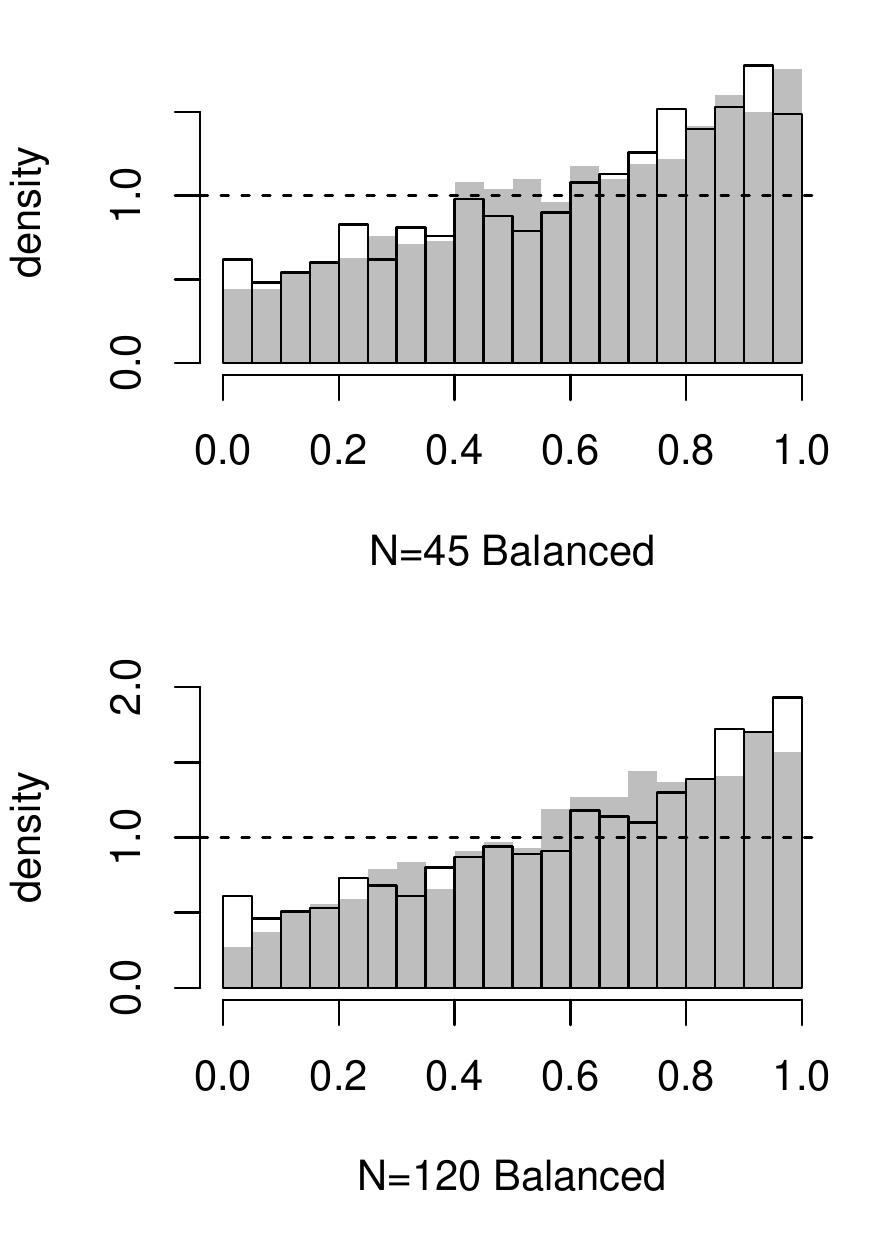}
\label{fig:subfigure1-S}}
\quad
\subfigure[Unbalanced experiments, case S2]{%
\includegraphics[width=0.31\textwidth]{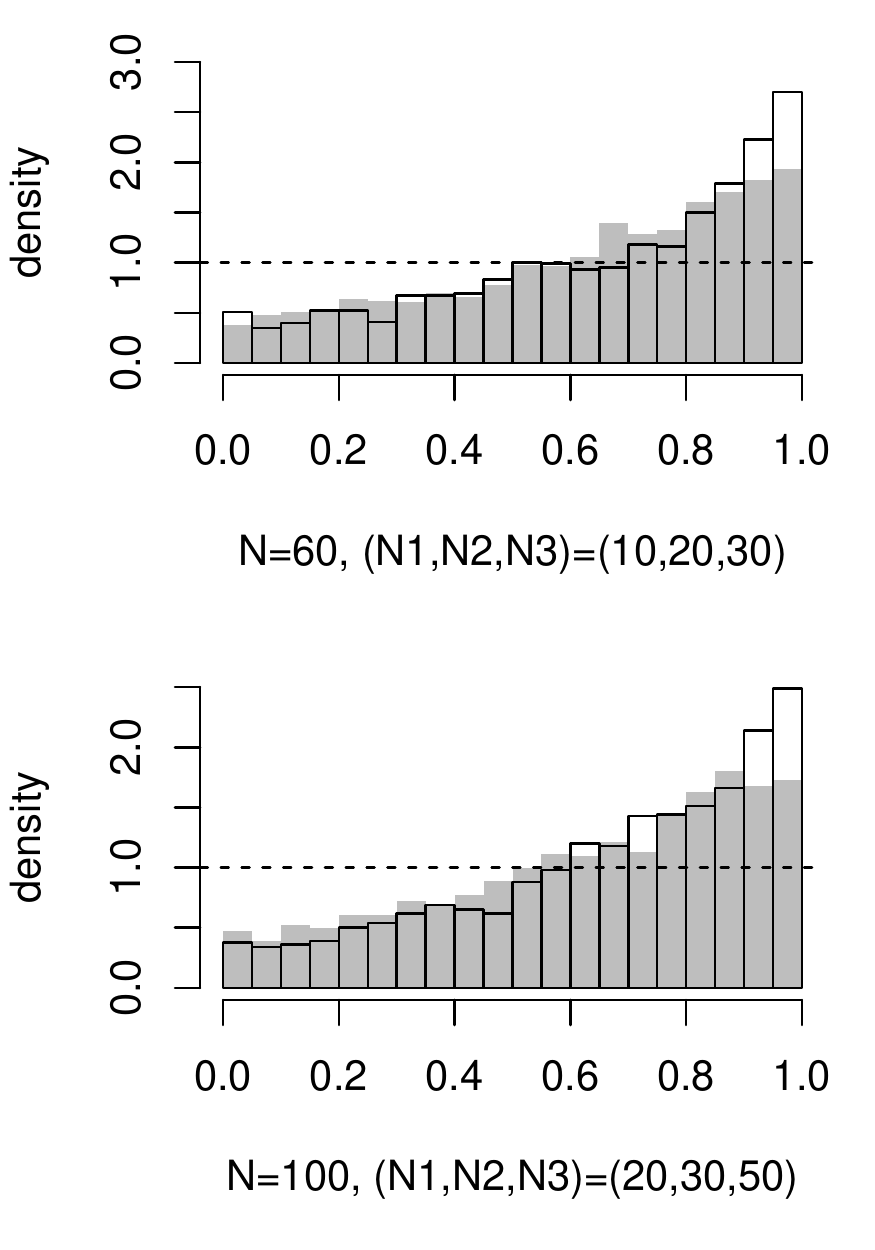}
\label{fig:subfigure2-S}}
\subfigure[Unbalanced experiments, case S3]{%
\includegraphics[width=0.31\textwidth]{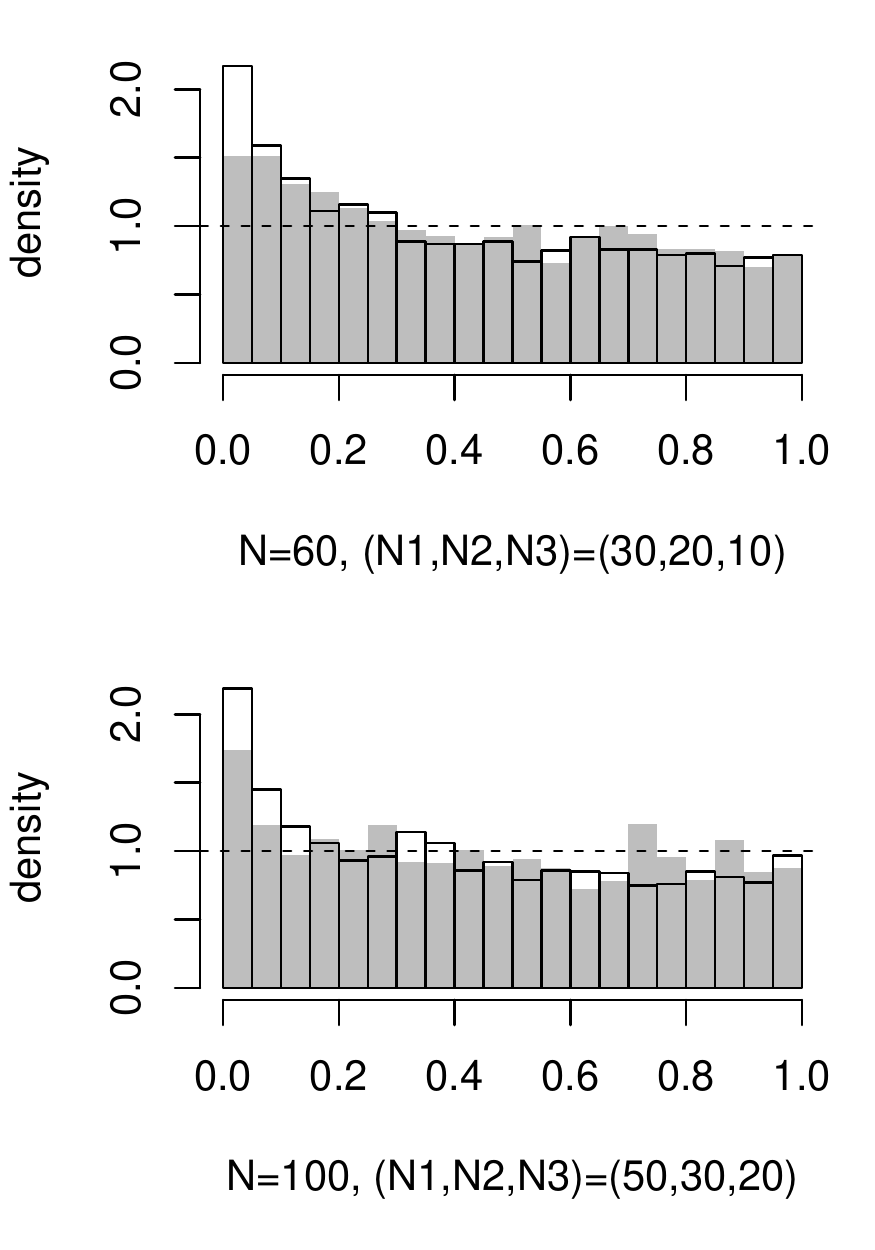}
\label{fig:subfigure3-S}}
\caption{Histograms of the $p$-values under $\HN$ based on the Fisher randomization tests using $X^2$, with grey histogram and white histograms for the first and second sub-cases.}
\label{fig:figure-S}
\end{figure}

\subsection{Type I error of the \FRT using $X^2$}
\label{subsec::x2-S}

We follow \S \ref{subsec::x2}, generate the same data as \S \ref{subsec::f-S}, and obtain the same conclusions about the Fisher randomization test using $X^2$, because Figures \ref{fig:figure_weight} and \ref{fig:figure_weight-S} exhibit the same pattern. All the Monte Carlo standard errors of the rejection rates below are close but no larger than $0.005.$

In Figure \ref{fig:subfigure1_weight-S}, for case (S1.1), the rejection rates are $0.034$ and $0.018$, and for case (S1.2), the rejection rates are $0.048$ and $0.029$, for sample sizes $N=45$ and $N=120$ respectively.
In Figure \ref{fig:subfigure2_weight-S}, for case (S2.1), the rejection rates are $0.032$ and $0.035$, and for case (S2.2), the rejection rates are $0.025$ and $0.036$, for sample sizes $N=45$ and $N=120$ respectively.
In Figure \ref{fig:subfigure3_weight-S}, for case (S3.1), the rejection rates are $0.060$ and $0.062$, and for case (S3.2), the rejection rates are $0.054$ and $0.044$, for sample sizes $N=45$ and $N=120$ respectively. 
This, coupled with Figure \ref{fig:figure-S}, agrees with our theory that the \FRT using $X^2$ can control type I error under $\HN$ better than using $F$.

\begin{figure}[t]
\centering
\subfigure[Balanced experiments, case S1]{%
\includegraphics[width=0.31\textwidth]{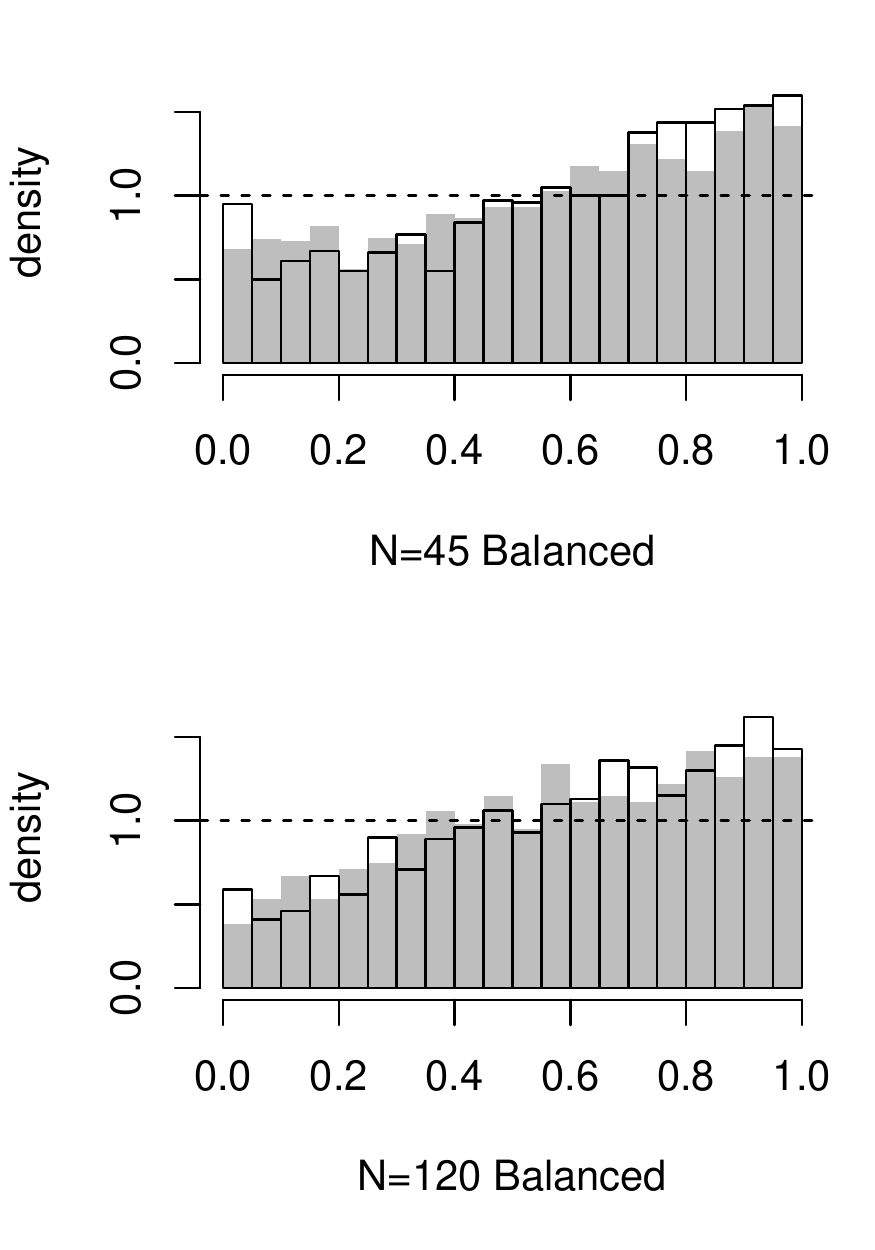}
\label{fig:subfigure1_weight-S}}
\quad
\subfigure[Unbalanced experiments, case S2.]{%
\includegraphics[width=0.31\textwidth]{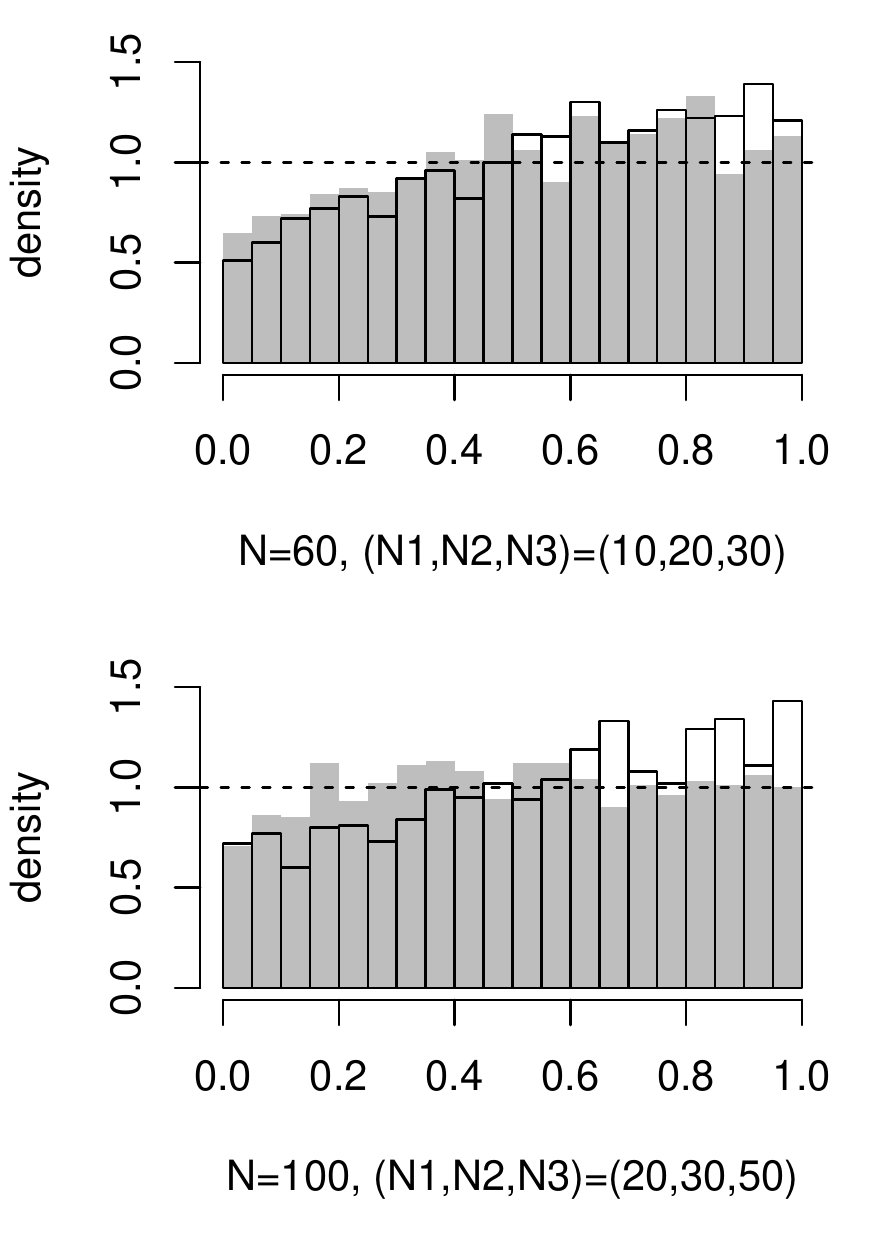}
\label{fig:subfigure2_weight-S}}
\subfigure[Unbalanced experiments, case S3]{%
\includegraphics[width=0.31\textwidth]{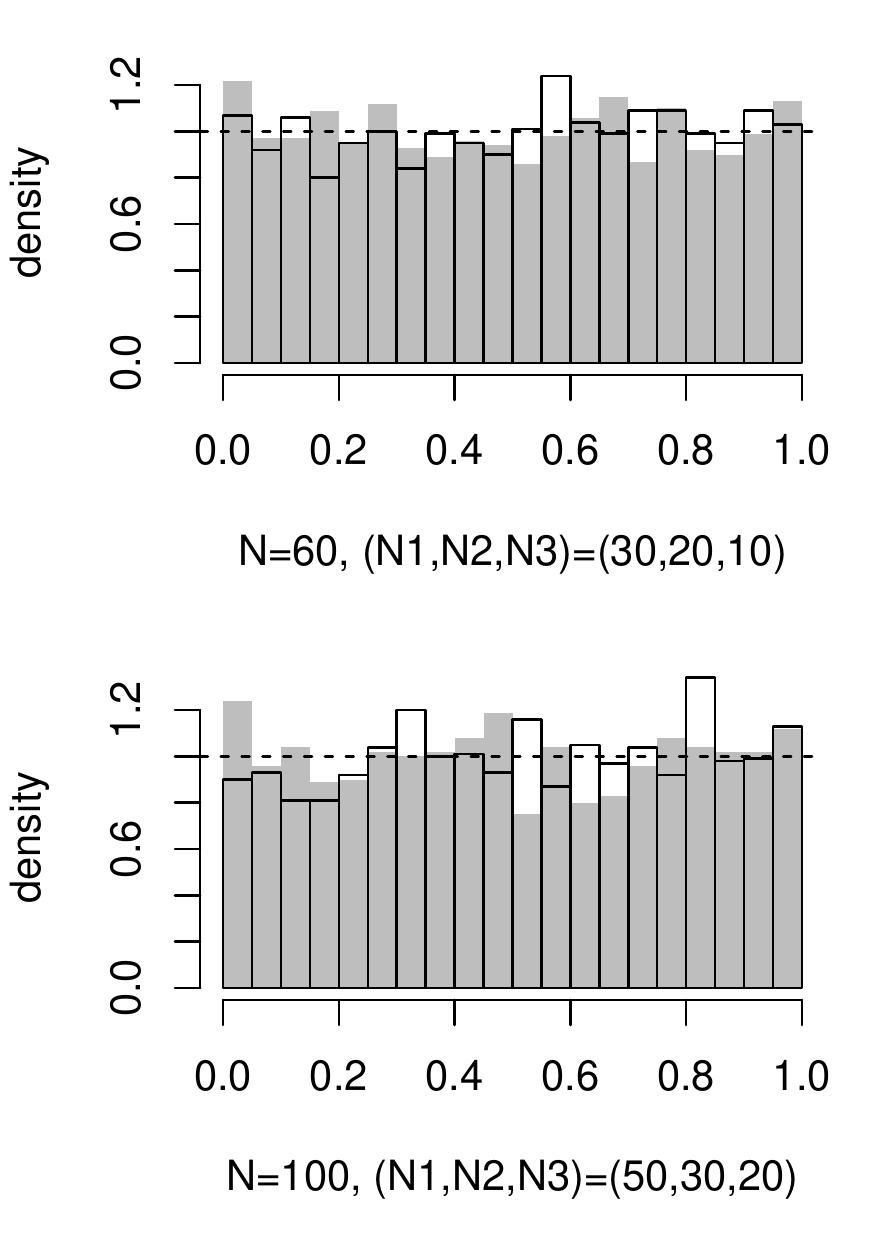}
\label{fig:subfigure3_weight-S}}
\caption{Histograms of the $p$-values under $\HN$ based on the Fisher randomization tests using $X^2$, with grey histogram and white histograms for the first and second sub-cases.}
\label{fig:figure_weight-S}
\end{figure}

\subsection{Power comparison of the Fisher randomization tests using $F$ and $X^2$}

We follow \S \ref{subsec::power} to compare the powers of the Fisher randomization tests using $F$ and $X^2$. We consider the following cases and summarize the results in Figure \ref{fig:figure_power-S}.  

Case S4.
For balanced experiments with sample sizes $N=30$ and $N=45$, we generate potential outcomes from $Y_i(1)\sim \mathcal{E}$, $Y_i(2)\sim \mathcal{E}/0.7$, $Y_i(3)\sim \mathcal{E} / 0.5$. These potential outcomes are independently generated, and shifted to have means $(0,0.5,1)$.

Case S5.
For unbalanced experiments with sample sizes $(N_1, N_2, N_3) = (10,20,30)$ and $(N_1, N_2, N_3) = (20,30,50)$, we first generate $Y_i(1)\sim \mathcal{E}$ and standardize them to have mean zero, and we then generate $Y_i(2) = 3Y_i(1)+1$ and $Y_i(3) = 5Y_i(1) + 2$. In this case, $p_1<p_2<p_3$ and $S^2_{\cdot}(1) < S^2_{\cdot}(2) < S^2_{\cdot}(3).$

Case S6.
For unbalanced experiments with sample sizes $(N_1, N_2, N_3) = (30,20,10)$ and $(N_1, N_2, N_3) = (50,30,20)$, we generate potential outcomes the same as the above case S5. In this case, $p_1>p_2>p_3$ and $S^2_{\cdot}(1) <S^2_{\cdot}(2) < S^2_{\cdot}(3).$

When the sample sizes are positively associated with the variances of the potential outcomes, the \FRT using $F$ has larger power than that using $X^2$. However, when the treatment groups are balanced or when the sample sizes are negatively associated with the variances of the potential outcomes, the \FRT using $F$ has smaller power than that using $X^2$. We report the rejection rates below with all the Monte Carlo standard errors no larger than $0.01.$

For case S4, the rejection rates using $X^2$ and $F$ are $0.087$ and $0.066$ with sample size $N=30$, and $0.207$ and $0.198$ with sample size $N=45$. For case S5, the powers using $X^2$ and $F$ are $0.044$ and $0.106$ with sample size $N=60$, and $0.293$ and $0.729$ with sample size $N=100$. For case S6, the rejection rates using $X^2$ and $F$ are $0.211$ and $0.037$ with sample size $N=60$, and $0.578$ and $0.274$ with sample size $N=100$.

\begin{figure}[t]
\centering
\subfigure[Balanced experiments, case S4]{%
\includegraphics[width=0.31\textwidth]{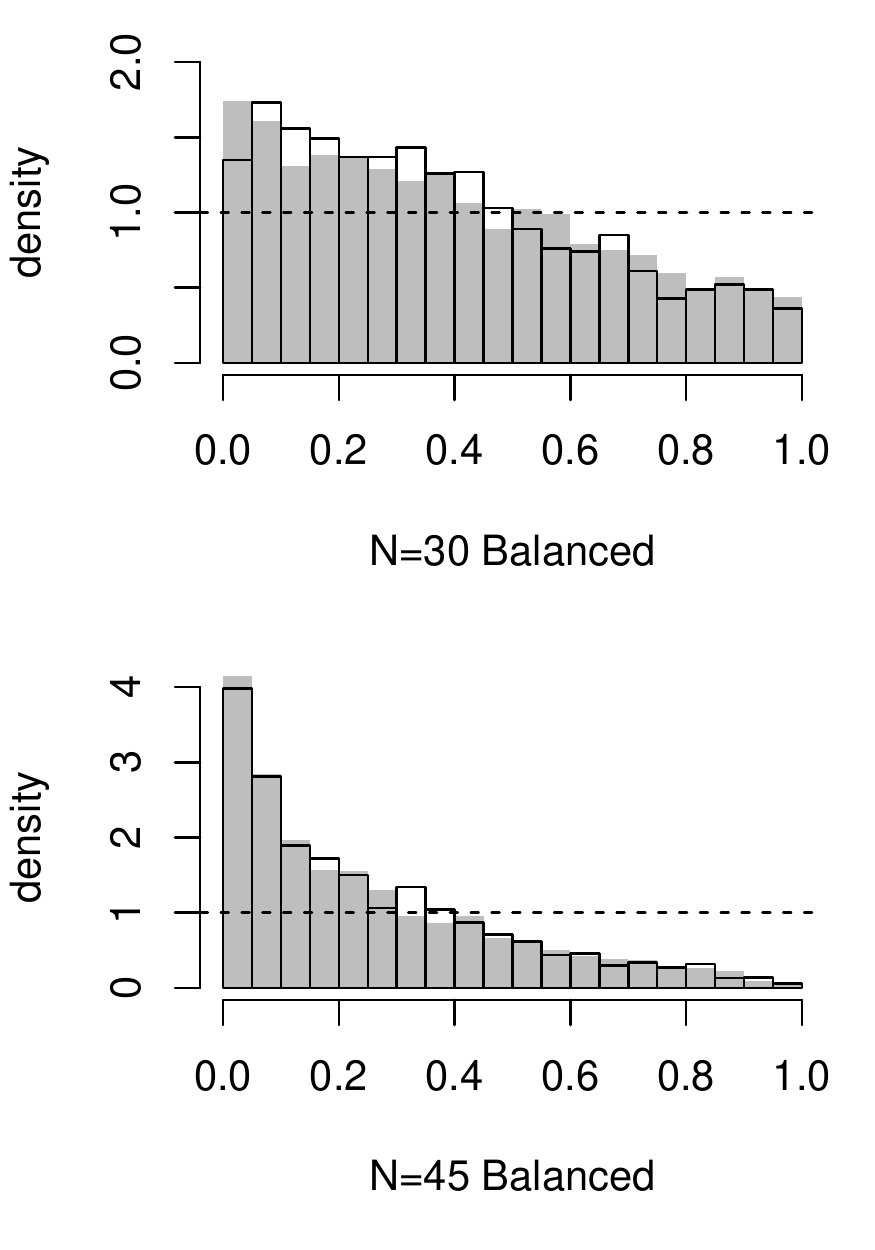}
\label{fig:subfigure1_power-S}}
\quad
\subfigure[Unbalanced experiments, case S5]{%
\includegraphics[width=0.31\textwidth]{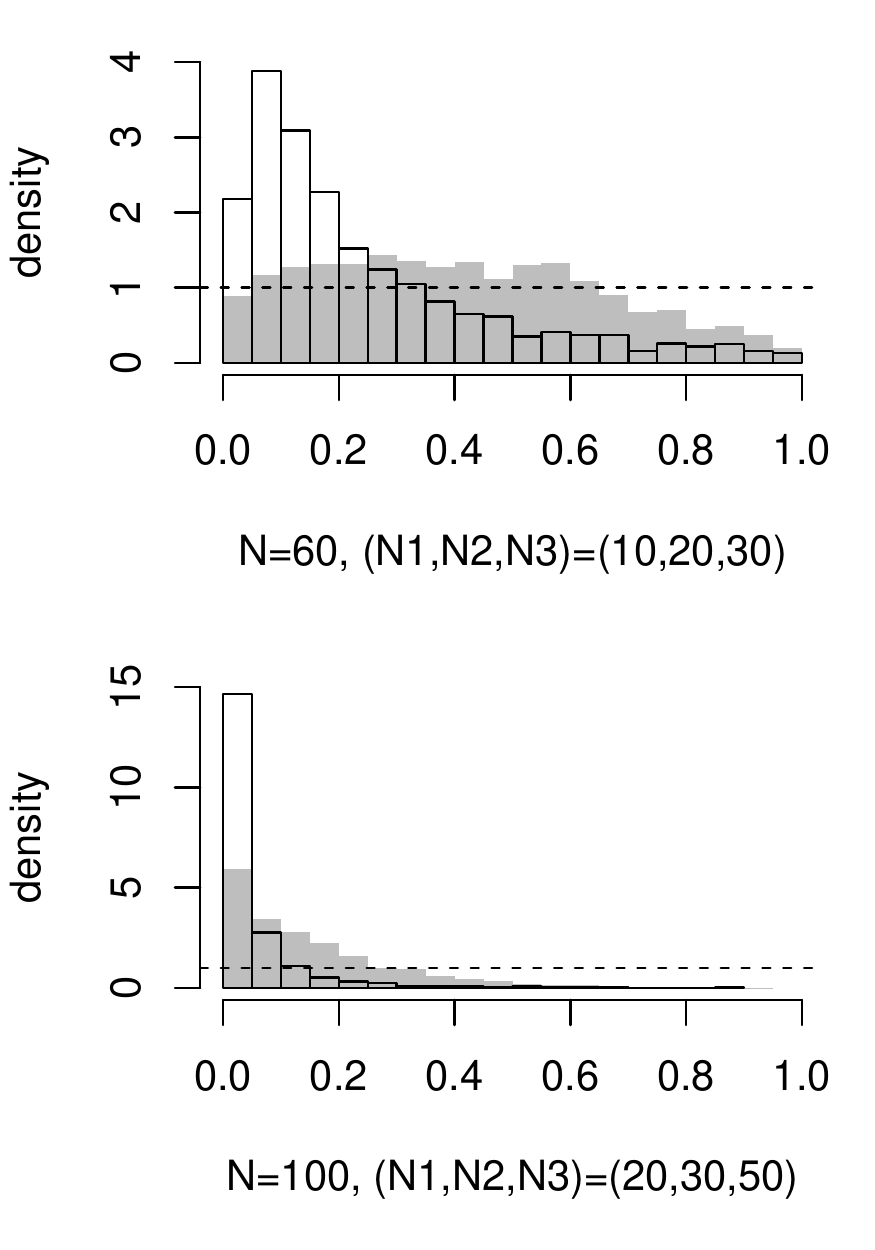}
\label{fig:subfigure2_power-S}}
\subfigure[Unbalanced experiments, case S6]{%
\includegraphics[width=0.31\textwidth]{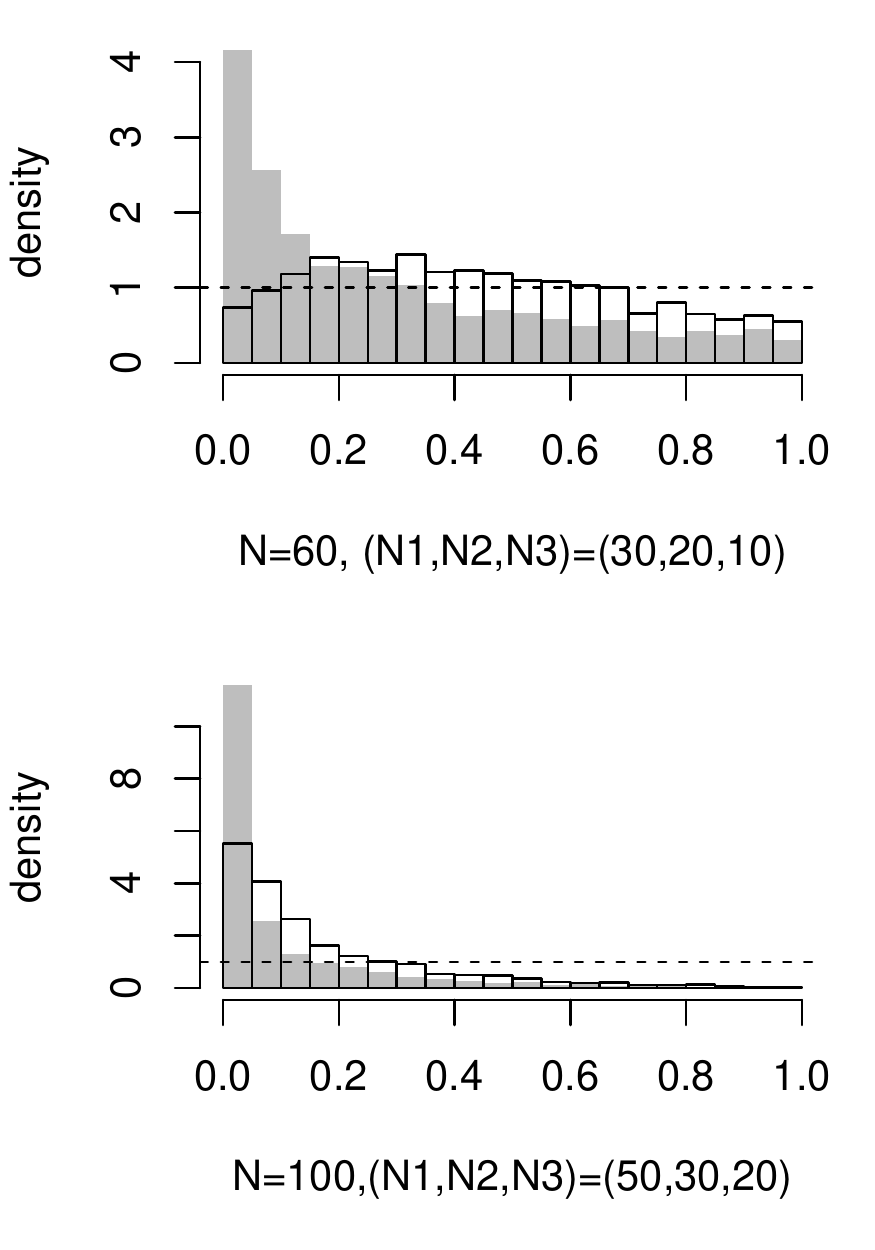}
\label{fig:subfigure3_power-S}}
\caption{Histograms of the $p$-values under alternative hypotheses based on the Fisher randomization tests using $F$ and $X^2$, with grey histograms for $X^2$ and white histograms for $F$.}
\label{fig:figure_power-S}
\end{figure}

\subsection{Finite sample evaluation of Corollary \ref{coro::equiv} with skewed outcomes}

We first generate log-normal potential outcomes $Y_i(1) \sim \exp\{  \mathcal{N}(0,1) \}$, $Y_i(2) \sim \exp\{  \mathcal{N}(1,1) \}$, and $Y_i(3) \sim \exp\{  \mathcal{N}(2,1) \}$, and then standard them to have equal finite population means $0$ and variances $1.$

Under $\HN$, the $p$-values of the \FRT using $F$ and $X^2$ are shown in Figure \ref{fig:homo}(a). With sample size $(N_1,N_2,N_3) = (10,10,10)$, the rejection rates using $X^2$ and $F$ are $0.012$ and $0.016$; with sample size $(10,15,20)$, the rejection rates are $0.016$ and $0.028$; with sample size $(20,15,10)$, the rejection rates are $0.006$ and $0.015$. The Monte Carlo standard errors are all close to but no larger than $0.004.$

%
%
%

Under alternative hypotheses, the $p$-values of the \FRT using $F$ and $X^2$ are shown in Figure \ref{fig:homo}(b). With sample size $(N_1,N_2,N_3) = (10,10,10)$, we shift the potential outcomes by constants $(0,0.5,1)$, and the rejection rates using $X^2$ and $F$ are $0.514$ and $0.512$; with sample size $(10,15,20)$, we shift the potential outcomes by constants $(0,0.2,0.5)$, and the rejection rates are $0.164$ and $0.215$; with sample size $(20,15,10)$, we shift the potential outcomes by constants $(0,0.2,0.5)$, and the rejection rates are $0.256$ and $0.179$. The Monte Carlo standard errors are all close but no larger than $0.011.$

In finite samples, we observe moderate difference between the Fisher randomization tests using $X^2$ and $F$ even with homoskedastic potential outcomes, although Corollary \ref{coro::equiv} ensures their asymptotic equivalence.

%
%
%

\begin{figure}[t]
\centering
\subfigure[$\HN$ holds]{%
\includegraphics[width=\textwidth]{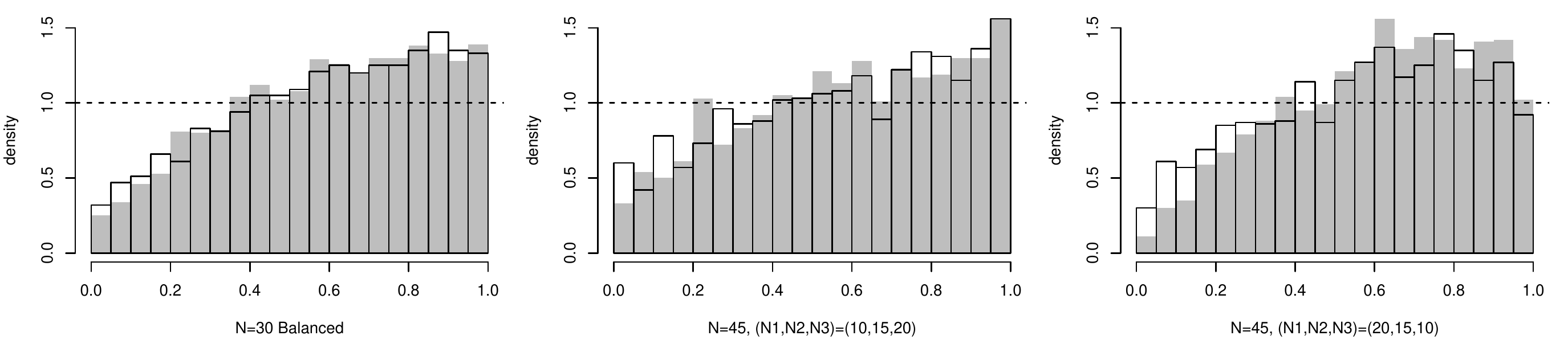}}
\quad
\subfigure[$\HN$ does not hold]{%
\includegraphics[width=\textwidth]{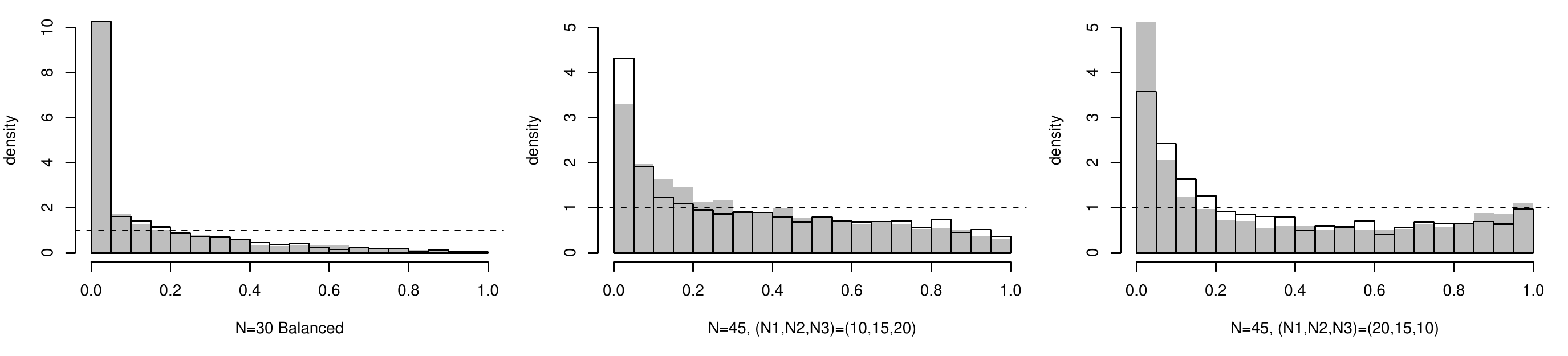}}
\caption{Histograms of the $p$-values under equal finite population variances based on the Fisher randomization tests using $F$ and $X^2$, with grey histograms for $X^2$ and white histograms for $F$.}
\label{fig:homo}
\end{figure}

\end{document}